\documentclass[11pt]{amsart}
\usepackage{graphicx}
\usepackage{color}
\usepackage{array}
\usepackage[margin=1in, centering]{geometry}
\usepackage{amsmath,amsthm}
\usepackage{amssymb,latexsym}
\usepackage{mathrsfs}
\usepackage{cancel}
\usepackage[initials,short-journals,short-months]{amsrefs}
\usepackage[pdfstartview={FitH}]{hyperref}
\usepackage{url}
\usepackage{tikz}
\usepackage{soul}
\usepackage{bbm}
\usepackage[FIGTOPCAP]{subfigure}

\theoremstyle{plain}
\newtheorem{theorem}{Theorem}
\newtheorem{lemma}{Lemma}[section]
\newtheorem{proposition}[lemma]{Proposition}
\newtheorem{corollary}[lemma]{Corollary}

\theoremstyle{definition}
\newtheorem{definition}[lemma]{Definition}
\newtheorem{assumption}{Assumption}
\newtheorem*{condition}{Condition}

\theoremstyle{remark}
\newtheorem{remark}[lemma]{Remark}
\newtheorem{example}[lemma]{Example}

\def \CSmooth(#1,#2){\mathcal{C}_{#1,#2}}

\def \Mgale(#1,#2){M_{#1}^{#2}}
\def \Ngale(#1,#2){\mathcal{N}_{#1}^{#2}}

\newcommand{\pt}{\mathsf{p}}

\newcommand{\Uone}[1]{U^{(1)}_{{#1}}}
\newcommand{\Utwo}[3]{U^{(2)}_{{#1},{#2},{#3}}}
\newcommand{\Utildetwo}[3]{\tilde{U}^{(2)}_{{#1},{#2},{#3}}}
\newcommand{\avg}[2]{{\rm Av}_{#2}\left[{#1}\right]}
\newcommand{\tc}{\tau_N}

\title[Local ergodicity in the exclusion process on an infinite weighted graph]{Local ergodicity in the exclusion process on an infinite weighted graph}
\author{Joe P.\@ Chen}
\address{Department of Mathematics, Colgate University, Hamilton, NY 13346, USA.
}
\email{jpchen@colgate.edu}
\urladdr{\url{http://math.colgate.edu/~jpchen}}
\thanks{Research partially supported by NSF grants DMS-1262929 (PI: Luke Rogers) and DMS-1613025 (PI: Alexander Teplyaev), and the Research Council of Colgate University.}

\date{\today}

\begin{document}


\renewcommand{\theequation}{\thesection.\arabic{equation}}
\numberwithin{equation}{section}

\maketitle

\begin{abstract}
We establish an abstract local ergodic theorem, under suitable space-time scaling, for the (boundary-driven) symmetric exclusion process on an increasing sequence of balls covering an infinite weighted graph. The proofs are based on 1-block and 2-blocks estimates utilizing the resistance structure of the graph; the moving particle lemma established recently by the author; and discrete harmonic analysis. Our ergodic theorem applies to any infinite weighted graph upon which random walk is strongly recurrent in the sense of Barlow, Delmotte, and Telcs; these include many trees, fractal graphs, and random graphs arising from percolation.

The main results of this paper are used to prove the joint density-current hydrodynamic limit of the boundary-driven exclusion process on the Sierpinski gasket, described in an upcoming paper with M.\@ Hinz and A.\@ Teplyaev.

\tableofcontents
\end{abstract}

\section{Introduction}


A main topic in statistical mechanics is the study of the emergence of collective phenomena arising from microscopic models. For instance, water is made out of $\sim 10^{23}$ water molecules, each interacting with one another through intermolecular forces (hydrogen bonds, Van der Waals forces). However, when one studies the macroscopic features of water, such as its density and the viscosity, the molecular details are ``averaged out'' and become less important. This is very much in the spirit of the \emph{law of large numbers} in probability theory. In a nutshell, one would like to justify, in a mathematically rigorous sense, how the hydrodynamics of water arises from interactions between the water molecules.

To embark upon this challenging problem, we need to first understand how to average the microscopic variables to obtain their macroscopic counterparts. In the physics literature this procedure often goes by the name ``renormalization'' or ``coarse-graining.'' If the underlying graph is the Euclidean lattice $\mathbb{Z}^d$, which are invariant under lattice translations and rotations, then one can use this symmetry to establish \emph{ergodicity} (mixing of space-time averages), which then justifies the aformentioned replacement by averages. In fact many existing proofs in the literature take direct advantage of the ergodicity under spatial translations and rotations.

That said, in real-life applications there are uncountably many natural or artificial networks which are not Euclidean lattices: some are trees, other have self-similar structures, and still others---\emph{random graphs}---have edges which may be present or absent with certain probabilities. On these networks, translational or rotational invariance is broken, so in order to establish limit theorems one is compelled to find alternative mechanisms which generate ergodicity.


The main goal of this paper is to establish an abstract local ergodic (or coarse graining) theorem for one of the most commonly studied interacting particle systems---the symmetric exclusion process---on an infinite connected weighted graph. See \S\ref{sec:exclusion} for the definition of the exclusion process, \cite{AldousFill, Spitzer, IPSStFlour} for introductory accounts, \cite{KipnisLandim, LiggettBook, Spohn} for technical backgrounds, and \cite{ABDS13, BDGJL15} and references therein for connections with non-equilibrium statistical mechanics. Throughout the paper, \emph{no assumption is made about the spatial symmetries of the underlying graph.} To effect the coarse graining argument, we use inputs from (discrete) potential theory and harmonic analysis (such as hitting estimates of random walks, harmonic functions, and Dirichlet energy). Upon establishing our main results, we will verify that our assumptions are satisfied on all \emph{(very) strongly recurrent weighted graphs}, in the sense of Barlow \cite{BarlowValues, BCK05}, Delmotte \cite{Delmotte}, and Telcs \cites{Telcs01, Telcs01_2}.

The present paper is the second of a four-part series establishing the hydrodynamic limit of the (boundary-driven) exclusion process on fractals, which has been summarized as a short review in \cite{SSEPreview}. The first part \cite{ChenMPL}, on the moving particle lemma for the exclusion process on a weighted graph, plays a crucial role in the proof of the local ergodic theorems in this paper. In turn, the results of this paper are used to establish hydrodynamic limit theorems for the exclusion process on the Sierpinski gasket in \cite{ChenTeplyaevSGHydro}. Since the scaling limit is a solution to a nonlinear PDE on a singular space, it behooves us to address issues of its existence, uniqueness, and regularity. These are described in \cite{CHTPDE}.

\section{Setup and main results} \label{sec:main}

Throughout the paper we assume that the undirected graph $G=(V,E)$ is connected and locally finite. 
Connected means that for any $x, y\in V$, there exists a sequence $\{x_0=1, x_1, \cdots, x_{n-1}, x_n=y\}$ in $V$ such that $x_{i-1} x_i \in E$ for all $i=1,2,\cdots, n$.
Locally finite means that the degree of every vertex is finite.
When we endow a direction upon an edge $e\in E$, its tail vertex will be denoted $\underline{e}$, and the head vertex $\overline{e}$.

In what follows, given a denumerable set $\Lambda$, $|\Lambda|$ stands for the cardinality of $\Lambda$. We denote the average of $g: \Lambda\to\mathbb{R}$ over $\Lambda$ by $\avg{g}{\Lambda} := \frac{1}{|\Lambda|}\sum_{z\in \Lambda} g(z)$. For each $\alpha\in [0,1]$ (resp.\@ each function $\gamma: \Lambda \to [0,1]$), let $\nu_\alpha^\Lambda$ (resp.\@ $\nu_{\gamma(\cdot)}^\Lambda$) denote the product Bernoulli measure on $\{0,1\}^\Lambda$ with marginal $\nu_\alpha^\Lambda\left(\{\eta: \eta(x)=1\}\right)=\alpha$ (resp.\@ $\nu_{\gamma(\cdot)}^\Lambda\left(\{\eta:\eta(x)=1\}\right)=\gamma(x)$) for each $x\in \Lambda$.

Given a Borel measure $\mu$ and a function $h\in L^1(\mu)$, we will adopt the shorthand $\mu[h] := \int\, h\,d\mu$ . 

Unless otherwise noted, the capitalized $C$ denotes a positive constant which may change from line to line. The dependence of the constants is indicated in subscripts, \emph{e.g.} $C_{\alpha_1, \alpha_2, \cdots}$ depends on $\alpha_1, \alpha_2, \cdots$. 
 
\subsection{Random walk on a weighted graph}

Consider a locally finite connected graph $\Gamma= (V(\Gamma), E(\Gamma))$ endowed with conductances ${\bf c} = (c_{xy})_{xy\in E(\Gamma)}$, where $c_{xy} > 0$. The pair $(\Gamma, {\bf c})$ is called a \emph{weighted graph}. For $x\in V(\Gamma)$ let $c_x = \sum_{y\sim x} c_{xy}$. The conductances induce a measure $\mathcal{V}$ on $V(\Gamma)$ given by 
$
\mathcal{V}(A) = \sum_{y\in A} c_x
$
for $A\subset V(\Gamma)$. 

The symmetric random walk process on $(\Gamma, {\bf c})$ is a Markov chain on $V(\Gamma)$ with transition probability
\[
p(x,y) = \left\{\begin{array}{ll} c_{xy}/c_x, & \text{if}~x\sim y,\\ 0, &\text{else}.\end{array}\right.
\]
A standard fact is that this process is reversible w.r.t.\@ the measure $\mathcal{V}$. The corresponding Dirichlet energy is
\[
\mathcal{E}^{\rm el}_{(\Gamma,{\bf c})}(f) = \sum_{xy\in E} c_{xy} [f(x)-f(y)]^2,\quad f: V(\Gamma)\to\mathbb{R}.
\]

Let $d$ be the graph metric on $\Gamma$, and for each $x\in V(\Gamma)$ and each $r\in \mathbb{N}_0$, let $B(x,r) := \{y\in V(\Gamma): d(x,y) < r\}$ be the open ball of radius $r$ centered at $x$. We set the volume of the ball centered at $x$ by
$
\mathcal{V}(x, r) = \mathcal{V}(B(x,r))
$.
Next, the hitting time of a set $A\subset V(\Gamma)$ by a random walk is denoted $T_A :=\inf\{t>0: X_t \in  A\}$.  We set the mean exit time from the ball centered at $x$ by
$
\mathcal{T}(x,r) = \mathbf{E}^x\left[T_{B(x,r)^c}\right]
$.
Finally, given two subsets $A_1, A_2 \subset V(\Gamma)$, the \textbf{effective resistance} between $A_1$ and $A_2$ is
\begin{align}
\label{Reff}
R_{\rm eff}^{(\Gamma, {\bf c})}(A_1,A_2) = \left(\inf\left\{ \mathcal{E}^{\rm el}_{(\Gamma,{\bf c})}(h) ~\bigg|~ h:V(\Gamma)\to\mathbb{R},~ h|_{A_1}=1,~h|_{A_2}=0\right\}\right)^{-1},
\end{align}
with the convention that $\inf \emptyset = \infty$.

\subsection{Exclusion process on a weighted graph} \label{sec:exclusion}

The symmetric exclusion process (SEP) on $(\Gamma, {\bf c})$ is a continuous-time Markov chain $(\eta_t)_{t\geq 0}$ on $\{0,1\}^{V(\Gamma)}$ with infinitesimal generator
\begin{align}
\left(\mathcal{L}_{(\Gamma, {\bf c})}^{\rm EX} f\right)(\eta) = \sum_{xy\in E(\Gamma)} c_{xy} (\nabla_{xy} f)(\eta), \quad f: \{0,1\}^{V(\Gamma)} \to \mathbb{R},
\end{align}
where $(\nabla_{xy} f)(\eta) = f(\eta^{xy})-f(\eta)$
and
\begin{align}
\eta^{xy}(z) = \left\{\begin{array}{ll}\eta(y),& \text{if}~z=x, \\ \eta(x), & \text{if}~z=y, \\ \eta(z), & \text{otherwise}.\end{array} \right.
\end{align}
Informally speaking, one starts with a configuration $\zeta$ in which $k$ vertices are occupied with a particle, and the remaining vertices are empty. All particles are deemed indistinguishable. A transition from $\zeta$ to $\zeta^{xy}$ occurs with rate $c_{xy}$ if and only if one of the vertices $\{x,y\}$ is occupied and the other is empty.

There are two key properties of the SEP. First, the total number of particles is conserved in the process. Second, the process is reversible with respect to any constant-density product Bernoulli measure $\nu_\alpha$ on $\{0,1\}^{V(\Gamma)}$, $\alpha\in [0,1]$, which has marginal $\nu_\alpha\{\zeta: \zeta(x)=1\}=\alpha
$ for all $x\in V(\Gamma)$. 



\begin{definition}
We say that $\phi: V(\Gamma) \times \{0,1\}^{V(\Gamma)} \to \mathbb{R}$ is a \emph{local function bundle for vertices} if there exists $r_\phi \in (0,\infty)$ such that for any $x\in V(\Gamma)$, $\phi_x :=\phi(x,\cdot)$ depends only on $\{\eta(z): z\in B(x,r_\phi)\}$ Likewise, we say that $\phi: E(\Gamma) \times \{0,1\}^{V(\Gamma)} \to \mathbb{R}$ is a \emph{local function bundle for edges} if there exists $r_\phi\in (0,\infty)$ such that for any $e= (\underline{e},\overline{e}) \in E(\Gamma)$, $\phi_e:= \phi(e,\cdot)$ depends only on $\{\eta(z) : z\in B(\underline{e}, r_\phi)\}$.
\end{definition}

\begin{example}
We list three examples of local function bundles.
\begin{enumerate}
\item $\phi(x,\eta) = \eta(x)$.
\item $\phi(x,\eta) = \sum_{y: xy\in E(\Gamma)} b_{xy} \eta(x)\eta(y)$, where $b_{xy} \in \mathbb{R}$. In practice, we consider the case $b_{xy} \equiv 1$ or the case $b_{xy}= c_{xy}$.
\item $\phi(e,\eta) = c_e \eta(\underline{e})\eta(\overline{e})$, where $c_e$ is the conductance of the edge $e$.
\end{enumerate}
The first example is rather trivial. The second and third examples are important for the proof of hydrodynamic limit of the exclusion process \cite{GPV88, KOV89, KipnisLandim}.
\end{example}

\begin{remark}
The terminology \emph{local function bundle} carries a similar notion as a \emph{cylinder function} on translation-invariant graphs. It is introduced in \cite{Tanaka}, although we do not assume that $\phi$ is invariant under the action of some infinite group. 
\end{remark}

Given a local function bundle for vertices $\phi$ and an $x\in V(\Gamma)$, we define the \emph{global average} of $\phi_x$ with respect to the product Bernoulli measure $\nu_\alpha$ on $\{0,1\}^{V(\Gamma)}$, $\alpha\in [0,1]$, by
\begin{align}
\Phi_x(\alpha) := \nu_\alpha[\phi_x].
 \end{align}
Since $\phi_x$ depends only on $\{\eta(x): x\in B(x,r_\phi)\}$, and $\nu_\alpha$ is product Bernoulli, it is direct to verify that $\alpha\mapsto \Phi_x(\alpha)$ is a Lipschitz function. 

Similarly, given a local function bundle for edges $\phi$ and an edge $e\in E(\Gamma)$, we define the \emph{global average} of $\phi_e$ with respect to $\nu_\alpha$ by
\begin{align}
\Phi_e(\alpha) := \nu_\alpha[\phi_e].
\end{align}
The map $\alpha \mapsto \Phi_e(\alpha)$ is Lipschitz by the same reasoning.


\subsection{Local ergodicity in the exclusion process} \label{sec:superexp}

Fix a vertex $o\in V(\Gamma)$ (``origin'') and a monotone increasing sequence of radii $(r_N)_{N\geq 1}$ with $r_1=1$ and $r_N \uparrow \infty$. This allows us to define an exhaustion of $\Gamma$ by finite graphs $(\Gamma_N)_{N\geq 1}$, where $V(\Gamma_N)= B(o,r_N)$ and $E(\Gamma_N) = \{xy\in E(\Gamma): x, y \in V(\Gamma_N)\}$. The edge conductances on $\Gamma_N$ are inherited from those on the mother graph $(\Gamma, {\bf c})$; we denote the corresponding finite weighted graph $(\Gamma_N, {\bf c})$.

Next, we introduce two increasing sequences of positive real numbers, $(\mathcal{V}_N)_{N\geq 1}$ and $(\mathcal{T}_N)_{N\geq 1}$. For applications, they will stand for, respectively, the sequence of \emph{mass} scales and \emph{time} scales; namely, $\mathcal{V}_N$ may stand for either $|B(o,r_N)|$ or $\mathcal{V}(o,r_N)$, and $\mathcal{T}_N$ may stand for the (extremal) expected time for a random walk starting in $B(o,r_N)$
to exit $B(o,r_N)$. Since the designations of these parameters vary with the weighted graph, we keep them as $\mathcal{V}_N$ and $\mathcal{T}_N$ for now, and defer their interpretations to \S\ref{sec:examples}. 


If $\phi: V(\Gamma)\times \{0,1\}^{V(\Gamma)}\to\mathbb{R}$ is a local function bundle for vertices, we set
\begin{align}
\label{eq:UNE}
U_{N,\epsilon}(x,\eta) := \phi_x(\eta)-  \Phi_x\left(\avg{\eta}{B(x,r_{\epsilon N})}\right), \qquad N\geq 1,~\epsilon\in [0,1].
\end{align}
Likewise, if $\phi: E(\Gamma)\times\{0,1\}^{V(\Gamma)}\to \mathbb{R}$ is a local function bundle for edges, we set
\begin{align}
\label{eq:UNE2}
U_{N,\epsilon}(e,\eta) := \phi_e(\eta)-  \Phi_e\left(\avg{\eta}{B(\underline{e},r_{\epsilon N})}\right), \qquad N\geq 1,~\epsilon\in [0,1].
\end{align}
Here and in what follows, $\epsilon N$ is to be understood as the integer part of $\epsilon N$.

For ease of notation, we will use $\pt$ to represent a vertex $x$ in the case of a local function bundle for vertices, or the tail vertex $\underline{e}$ of an edge $e$ in the case of a local function bundle for edges. So for instance, $B(\pt, r)$ stands for $B(x,r)$ in the former case, and $B(\underline{e},r)$ in the latter case. To abuse notation a bit further, we write $U_{N,\epsilon}(\pt, \eta)$  to denote \eqref{eq:UNE} in the former case, and \eqref{eq:UNE2} in the latter case.

We introduce the following assumptions.

\begin{assumption}
\label{ass:1}
\begin{align}
\liminf_{N\to\infty} \frac{\mathcal{T}_N}{\mathcal{V}_N} = \infty.
\end{align}
\end{assumption}

\begin{assumption}
\label{ass:2}
For each $x\in V(\Gamma)$ and each $y,z \in B(x,r_{\epsilon N})$,
\begin{align}
\label{eq:ass2}
\liminf_{\epsilon\downarrow 0} \liminf_{N\to\infty} \frac{\mathcal{T}_N}{\mathcal{V}_N} \left(R_{\rm eff}^{(\Gamma_N,{\bf c})}(y,z)\right)^{-1} = \infty.
\end{align}
\end{assumption}

Here is our main theorem.

\begin{theorem}[Local ergodicity in the exclusion process]
\label{thm:localrep}
Let $\mathbb{P}^N_{\alpha}$ be the law of the symmetric exclusion process $(\eta_t^N)_{t\geq 0}$ with generator $\mathcal{T}_N \mathcal{L}^{\rm EX}_{(\Gamma_N, {\bf c})}$, started from the product Bernoulli measure $\nu_\alpha$ on $\{0,1\}^{V(\Gamma_N)}$. Under Assumptions \ref{ass:1} and \ref{ass:2}, for each $T>0$ and each $\delta>0$,
\begin{equation}
\label{supexp}
\limsup_{\epsilon\downarrow 0} \limsup_{N\to\infty} \sup_{\pt} \frac{1}{\mathcal{V}_N} \log \mathbb{P}^N_{\alpha} \left\{ \left|\int_0^T\, U_{N,\epsilon}(\pt,\eta^N_t)\,dt\right|>\delta\right\} = -\infty,
\end{equation}
where $\pt$ stands for $x$ (resp.\@ $e$) in the case of a local function bundle for vertices (reps.\@ for edges), and the supremum runs over all $x\in V(\Gamma_N)$ (resp.\@ over all $e\in E(\Gamma_N)$).
\end{theorem}

Assumptions \ref{ass:1} and \ref{ass:2} are satisfied on the so-called \emph{(very) strongly recurrent} weighted graphs, which include $\mathbb{Z}$; trees which support recurrent random walks; post-critically finite (p.c.f.\@) fractal graphs, such as the Sierpinski gasket graph; Sierpinski carpet graphs; and random graphs arising from percolation models. For the precise conditions and examples see \S\ref{sec:examples}.

As is known to experts in interacting particle systems, the proof of Theorem \ref{thm:localrep} relies upon the \emph{one-block estimate} and the \emph{two-blocks estimate}. The one-block estimate involves replacing spins by their averages over large microscopic blocks (of scale $j$), while the two-blocks estimate involves replacing spin averages over large microscopic blocks by spin averages over small macroscopic blocks (of scale $\epsilon N$). The claim is that both replacement costs vanish in the diffusive limit. We use a local version of these estimates, introduced in \cite{JLSLocal}, since our graphs generally lack translational invariance.

\begin{theorem}[Local one-block estimate]
\label{thm:1block}
Let $\pt$ stand for $x\in V(\Gamma)$ (resp.\@ $e\in E(\Gamma)$) in the case of a local function bundle for vertices (resp.\@ for edges), and $\{\Lambda_j(\pt)\}_{j\geq 1}$ be a sequence of increasing finite connected subsets of $\Gamma$ containing $\pt$, with $\lim_{j\to\infty} |\Lambda_j(\pt)|=\infty$. Define
\begin{equation}
\Uone{j}(\pt,\eta)  := \phi_\pt(\eta) - \Phi_\pt\left(\avg{\eta}{\Lambda_j(\pt)}\right), \quad \eta\in \{0,1\}^{V(\Gamma)}.
\end{equation}
Then under Assumption \ref{ass:1}, for each $T>0$ and each $\delta>0$,
\begin{align}
\label{1bp} \limsup_{j\to\infty} \limsup_{N\to\infty} \sup_{\pt} \frac{1}{\mathcal{V}_N} \log \mathbb{P}^N_\alpha \left[\left|\int_0^T \, \Uone{j}(\pt,\eta^N_t)\,dt\right|>\delta\right] &=-\infty,
\end{align}
where the supremum runs over all $x\in V(\Gamma_N)$ (resp.\@ over all $e\in E(\Gamma_N)$).
\end{theorem}

\begin{theorem}[Local two-blocks estimate]
\label{thm:2block}
Assume the hypothesis of Theorem \ref{thm:1block}, and let
\begin{align}
\Utwo{N}{\epsilon}{j}(\pt,\eta) := \Phi_\pt\left(\avg{\eta}{\Lambda_j(\pt)}\right) -\Phi_\pt\left(\avg{\eta}{B(\pt, r_{\epsilon N})}\right),\quad \eta\in \{0,1\}^{V(\Gamma)}.
\end{align}
Then under Assumption \ref{ass:2}, for each $T>0$ and each $\delta>0$,
\begin{align}
\label{2bp} \limsup_{j\to\infty} \limsup_{\epsilon\downarrow 0} \limsup_{N\to\infty} \sup_\pt \frac{1}{\mathcal{V}_N} \log \mathbb{P}^N_\alpha \left[\left|\int_0^T \, \Utwo{N}{\epsilon}{j}(\pt,\eta^N_t) \,dt \right|>\delta \right] &=-\infty.
\end{align}
\end{theorem}

The one-block estimate, proved in \S\ref{sec:1block}, is fairly standard. In contrast, the proof of the two-blocks estimate in its present form is new. We use a version of the moving particle lemma, which appeared in \cite{ChenMPL}, valid on any finite weighted graph; see Proposition \ref{prop:MPL} below. This lemma then serves as the starting point of a coarse-graining argument which goes through by virtue of the separation of the two scales $j$ and $\epsilon N$. Our proof does not require precise geometric control on the averaging blocks. See \S\ref{sec:2blocks} for details.

\begin{proof}[Proof of Theorem \ref{thm:localrep}, assuming Theorems \ref{thm:1block} and \ref{thm:2block}]
Observe that
\begin{align}
\label{UDecomp} U_{N,\epsilon}(\pt,\eta)= \Uone{j}(\pt,\eta) + \Utwo{N}{\epsilon}{j}(\pt,\eta).
\end{align}

Set the random variables $X^{(1)}_N = \left|\int_0^T \, \Uone{j}(\pt,\eta^N_t)\,dt\right|$ and $X^{(2)}_N = \left|\int_0^T \, \Utwo{N}{\epsilon}{j}(\pt,\eta^N_t)\,dt\right|$. Observe that for any $\delta>2\delta'>0$,
\begin{align*}
\left\{\left|\int_0^T\, U_{N,\epsilon}(\pt, \eta^N_t)\,dt\right|>\delta\right\}  \subset \{X^{(1)}_N+X^{(2)}_N > \delta\} &\subset \{X^{(1)}_N \leq \delta',~X^{(2)}_N \leq \delta'\}^c \\
& = \{X^{(1)}_N>\delta'\} \cup \{X^{(2)}_N>\delta'\}.
\end{align*}
So by the union bound,
\begin{align*}
\mathbb{P}^N_{\eta_0^N}\left[X^{(1)}_N+X^{(2)}_N>\delta\right] \leq \mathbb{P}^N_{\eta_0^N}\left[X^{(1)}_N>\delta'\right] + \mathbb{P}^N_{\eta_0^N}\left[X^{(2)}_N>\delta'\right] =: p^{(1)}_N + p^{(2)}_N.
\end{align*}
Furthermore (see \emph{e.g.\@} \cite{KipnisLandim}*{(A.II.3.2)}), if $(a_N)_N$ is an increasing sequence of real numbers with $a_N \uparrow \infty$, and $(b_N)_N$ and $(d_N)_N$ are two sequences of positive real numbers, then
\begin{equation}
\label{px1x2}
\limsup_{N\to\infty} \frac{1}{a_N} \log(b_N + d_N) \leq \max \left(\limsup_{N\to\infty} \frac{1}{a_N}\log b_N, \limsup_{N\to\infty} \frac{1}{a_N} \log d_N\right).
\end{equation}
Setting $a_N=\mathcal{V}_N$, $b_N=p^{(1)}_N$, and $d_N=p^{(2)}_N$, we deduce that
\begin{align}
\label{qx1x2}
\limsup_{N\to\infty} \frac{1}{\mathcal{V}_N}\log\left(p^{(1)}_N + p^{(2)}_N\right) \leq \max\left(\limsup_{N\to\infty}\frac{1}{\mathcal{V}_N} \log p^{(1)}_N , \limsup_{N\to\infty} \frac{1}{\mathcal{V}_N}\log p^{(2)}_N \right).
\end{align}
Now take the $\epsilon\downarrow 0$ limit followed by the $j\to\infty$ limit on both sides of (\ref{qx1x2}), and apply (\ref{1bp}) and (\ref{2bp}) to deduce (\ref{supexp}).
\end{proof}

\subsection{Local ergodicity in the boundary-driven exclusion process} \label{sec:mainboundary}

To make connections with rigorous study of non-equilibrium statistical mechanics \cites{ABDS13, BodineauLagouge, BDAdditivity, BDGJL03, BDGJL07, BDGJL15}, we also consider the boundary-driven version of the exclusion process on a weighted graph. Informally speaking, we modify the symmetric exclusion process  defined in \S\ref{sec:exclusion} by introducing ``boundary reservoirs'' so that particles can jump from the bulk to the edge and vice versa. See Figure \ref{fig:BD}. By introducing the reservoirs, the stochastic process is generally no longer reversible, and therefore serves as a candidate model for stochastic dynamics out of equilibrium.

\begin{figure}
\centering
\includegraphics[width=0.6\textwidth]{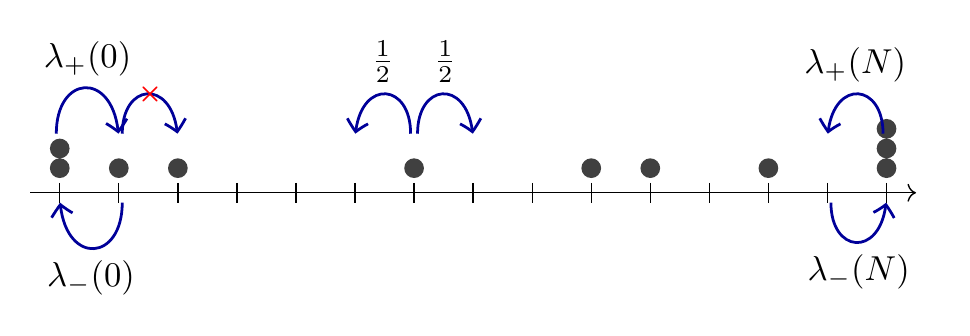}
\caption{The boundary-driven exclusion process on $\{0, 1,\cdots, N\}$, with reservoirs at $0$ and $N$.}
\label{fig:BD}
\end{figure}

Given a connected weighted graph $(G, {\bf c})$, we declare a nonempty subset $\partial V \subset V(G)$ to be the boundary set. For simplicity and without loss of generality, we assume that $c_{aa'}=0$ for all $a,a' \in \partial V$. Let $(\eta_t)_{t\geq 0}$ be a continuous-time Markov chain on $\{0,1\}^{V(G)}$ whose infinitesimal generator is 
\begin{align}
\label{eq:gen}
\mathcal{L}^{\rm bEX}_{(G,{\bf c})} = \mathcal{L}^{\rm EX}_{(G,{\bf c})} + \mathcal{L}^{\rm b}_{\partial V},
\end{align}
where the boundary generator reads
\begin{align}
(\mathcal{L}^b_{\partial V} f)(\eta)  = \sum_{a\in \partial V} [\lambda_-(a) \eta(a) + \lambda_+(a) (1-\eta(a))] [ f(\eta^a)-f(\eta)], \quad f: \{0,1\}^{V(G)} \to\mathbb{R},
\end{align}
with $\lambda_+(a) \in \mathbb{R}_+$ (resp.\@ $\lambda_-(a) \in \mathbb{R}_+$) representing the transition rate from $\eta(a)=0$ to $\eta(a)=1$ (resp.\@ from $\eta(a)=1$ to $\eta(a)=0$) with all other $\eta(y)$ intact, and
\begin{align}
\eta^a(z) = \left\{\begin{array}{ll} 1-\eta(a),& \text{if}~z=a,\\ \eta(z), & \text{otherwise.}\end{array}\right.
\end{align}

We impose the following conditions on the transition rates.

\begin{condition}[E] \label{ass:boundedaway}
$\displaystyle \limsup_{N\to\infty} \frac{|V(\Gamma_N)|}{\mathcal{V}_N} < \infty$.
\end{condition}

\begin{condition}[BR] \label{ass:boundaryrate}
There exist $\gamma,\gamma'\in [1,\infty)$ such that for all $a\in \partial V$, 
$$
\gamma^{-1} \leq \frac{\lambda_+(a)}{\lambda_-(a)} \leq \gamma \quad\text{and}\quad (\gamma')^{-1} \leq \frac{\lambda_+(a)}{c_a} \leq \gamma'.
$$
\end{condition}


As before, let $\{(\Gamma_N,{\bf c})\}_N$ be a sequence of finite weighted graphs exhausting $(\Gamma,{\bf c})$. To each $\Gamma_N$ we associate a boundary set $\partial V_N \subset V(\Gamma_N)$. (In practice one can take $\partial V_N = \partial B(o, r_N)$, though our proofs do not depend crucially on this choice aside from satisfying Assumption \ref{ass:boundaryscaling} below.) Let $\mathcal{L}^{\rm bEX}_{(\Gamma_N, {\bf c})} = \mathcal{L}^{\rm EX}_{(\Gamma_N, {\bf c})} + \mathcal{L}^b_{\partial V_N}$ denote the generator of the boundary-driven symmetric exclusion process on $(\Gamma_N, {\bf c})$ with rates $\{\lambda_\pm(a): a\in \partial V_N\}$. Let $\mu^N_{\rm inv}$ be the unique invariant measure of the process generated by $\mathcal{L}^{\rm bEX}_{(\Gamma_N, {\bf c})}$. Note that $\mu^N_{\rm inv}$ is generally \emph{not} a product measure, but one can compute its one-site marginal $\rho_N(x):= E_{\mu_{\rm inv}^N}[\eta(x)]$, $x\in V(\Gamma_N)$. It is not hard to show that the function $\rho_N$ takes value in $\left[\frac{1}{1+\gamma}, \frac{\gamma}{1+\gamma}\right]$ where $\gamma$ is as in Condition (\nameref{ass:boundaryrate}), and is harmonic (with respect to the graph Laplacian) on $V(\Gamma_N) \setminus \partial V_N$ subject to a Robin boundary condition on $\partial V_N$; see \S\ref{sec:dirichlet} for details.

To prove local ergodicity in the boundary-driven process, we need an additional assumption on the boundary sets and the corresponding rates. This is stated in the language of random walks (or electrical networks) following \emph{e.g.\@} \cite{DoyleSnell} or \cite{LyonsPeres}*{Ch.\@ 2}. Given a function $h: V(G)\to\mathbb{R}$ which is harmonic on $V(G)\setminus \partial V$, the \textbf{(electric) flow} of $h$ out of $a\in \partial V$ is given by
\begin{align}
\label{eq:flow}
i_h(a) = \sum_{y\in V(G)} c_{ay}[h(y)-h(a)].
\end{align}
In the sense of calculus on graphs, $i_h(a)$ can be also seen as the negative of the (discrete) \textbf{normal derivative} of $h$ at $a$.

\begin{assumption}
The sequence of boundary rates $(\{\lambda_\pm(a): a\in \partial V_N\})_N$ is chosen such that
\label{ass:boundaryscaling}
\begin{align}
\label{flowfinite}
  \limsup_{N\to\infty} \frac{\mathcal{T}_N}{\mathcal{V}_N} \sum_{a\in \partial V_N} |i_{\rho_N}(a)| <\infty
 \end{align}
 and
 \begin{align} 
 \label{energyfinite}
 \limsup_{N\to\infty} \frac{\mathcal{T}_N}{\mathcal{V}_N} \mathcal{E}^{\rm el}_{(\Gamma_N,{\bf c})}(\rho_N) < \infty.
\end{align}
\end{assumption}

This assumption says that the sum of the absolute values of the flows of $\rho_N$ on the boundary set, as well as the Dirichlet energy of $\rho_N$, remains finite in the scaling limit. It will be explained in Remark \ref{rem:trace} that
\[
i_{\rho_N}(a) = \sum_{b\in \partial V_N} \hat{c}_{a,b} [\rho_N(b)-\rho_N(a)], \quad a\in \partial V_N,
\]
and
\[
 \mathcal{E}^{\rm el}_{(\Gamma_N,{\bf c})}(\rho_N) = \frac{1}{2} \sum_{a\in \partial V_N} \sum_{b\in\partial V_N} \hat{c}_{a,b} [\rho_N(a)-\rho_N(b)]^2,
\]
where $\{\hat{c}_{a,b}>0:a,b\in \partial V_N\}$ represent the (asymmetric) conductances in the \emph{trace} of the symmetric random walk process on $(\Gamma_N, {\bf c})$ to the boundary $\partial V_N$. Thus \eqref{flowfinite} and \eqref{energyfinite} correspond to finiteness of, respectively, the $L^1$- and $L^2$-gradient norm in the (rescaled) trace process.

Let us note that Assumption \ref{ass:boundaryscaling} is implied by the condition
\[
\limsup_{N\to\infty} \frac{\mathcal{T}_N}{\mathcal{V}_N} \sum_{a\in V_N} \sum_{\substack{b\in V_N\\ b\neq a}} \hat{c}_{a,b} < \infty.
\]
Moreover, Assumption \ref{ass:1} and the finite flow condition \eqref{flowfinite} of Assumption \ref{ass:boundaryscaling} together imply that $\lim_{N\to\infty} \sum_{a\in \partial V_N} |i_{\rho_N}(a)|=0$, whence $\lim_{N\to\infty} \sup_{a \in \partial V_N} |i_{\rho_N}(a)|=0$. In turn, by the boundary condition on $\rho_N$, \emph{viz.\@} the second equation in \eqref{ODErho}, we obtain
\begin{align}
\label{eq:ratecoincide}
\lim_{N\to\infty} \sup_{a\in \partial V_N} \left| \rho_N(a) - \frac{\lambda_+(a)}{\lambda_+(a) + \lambda_-(a)} \right| =0.
\end{align}
This coincides with the original motivation for the boundary-driven process that in the scaling limit, the mean particle density at the boundary point $a\in \partial V$ is determined solely by the ratio of the boundary hopping rates $\lambda_+(a)/\lambda_-(a)$.

\begin{theorem}[Local ergodicity in the boundary-driven exclusion process]
\label{thm:localergodicboundary}
Let $\mathbb{P}^N_{\eta_0^N}$ be the law of the boundary-driven exclusion process $(\eta_t^N)_{t\geq 0}$ with infinitesimal generator $\mathcal{T}_N \mathcal{L}^{\rm bEX}_{(\Gamma_N, {\bf c})}$, started from the initial configuration $\eta_0^N$ on $\{0,1\}^{V(\Gamma_N)}$. Under Conditions (\nameref{ass:boundedaway}) and (\nameref{ass:boundaryrate}), as well as Assumptions \ref{ass:1}, \ref{ass:2}, and \ref{ass:boundaryscaling},
 for each $T>0$ and each $\delta>0$,
\begin{align}
\label{supexpboundary}
\limsup_{\epsilon\downarrow 0} \limsup_{N\to\infty} \sup_\pt \frac{1}{\mathcal{V}_N} \log \mathbb{P}^N_{\eta_0^N} \left\{ \left|\int_0^T\, U_{N,\epsilon}(\pt,\eta^N_t)\,dt\right|>\delta\right\} = -\infty.
\end{align}
Also, for each $T>0$, $\delta>0$, and continuous function $G:[0,T]\to\mathbb{R}$,
\begin{align}
\label{supexpboundary2}
\limsup_{N\to\infty} \sup_{a\in \partial V_N} \frac{1}{\mathcal{V}_N} \log\mathbb{P}^N_{\eta_0^N}\left\{\left|\int_0^T\, G(t)\left[\eta_t^N(a)-\frac{\lambda_+(a)}{\lambda_+(a)+\lambda_-(a)} \right]\,dt \right| >\delta\right\}=-\infty.
\end{align}
\end{theorem}

In proving Theorem \ref{thm:localergodicboundary}, we adapt the strategy of \cite{BDGJL03} by introducing a ``mock'' product measure whose marginals are $\rho_N(x)$, the one-site marginals of $\mu_{\rm inv}^N$, and run the one-block and two-blocks estimates with respect to this mock measure. Details are described in \S\ref{sec:boundary}.

The key outstanding question from our work is whether local ergodicity can be established without Assumption \ref{ass:1}, that is, on non-translationally-invariant graphs which support transient random walks (analogs of $\mathbb{Z}^d,~d\geq 3$), or even weakly recurrent random walks (analogs of $\mathbb{Z}^2$). This and other open questions are discussed in \S\ref{sec:open}.

\section{One-block estimate: Proof of Theorem \ref{thm:1block}} \label{sec:1block}

The arguments in this section mirror those given in \cite{Jara}*{\S 5.1} and \cite{JLSLocal}*{\S 6.1}.

Given a subset $\Lambda \subset V(\Gamma)$, we define the generator
\begin{align}
(\mathcal{L}^{\rm EX}_{(\Lambda,{\bf c})} f)(\eta) = \sum_{\substack{xy\in E(\Gamma) \\ x,y \in \Lambda}} c_{xy} (\nabla_{xy} f)(\eta), \quad f: \{0,1\}^{V(\Gamma)} \to\mathbb{R}. 
\end{align}
It is direct to verify that the product Bernoulli measure $\nu_\alpha$ on $\{0,1\}^{V(\Gamma)}$ is a reversible invariant measure for $\mathcal{L}^{\rm EX}_{(\Lambda,{\bf c})}$, whence $\nu_\alpha\left[f \left( -\mathcal{L}^{\rm EX}_{(\Lambda,{\bf c})} f\right)\right] \geq 0$ for all $f\in L^2(\nu_\alpha)$. 

We have the following trivial comparison of the Dirichlet forms.
\begin{lemma}
\label{lem:DFEst1}
Suppose $\Lambda_1 \subset \Lambda_2 \subset \Gamma$. Then for every $f \in L^2(\nu_\alpha^{\Lambda_2})$, 
\begin{align}
\nu_\alpha^{\Lambda_2}\left[ f\left(-\mathcal{L}^{\rm EX}_{(\Lambda_1,{\bf c})}f\right)\right] \leq \nu_\alpha^{\Lambda_2} \left[ f \left( -\mathcal{L}^{\rm EX}_{(\Lambda_2,{\bf c})}f\right)\right].
 \end{align}
\end{lemma}
\begin{proof}
Since $\nu_\alpha^{\Lambda_2}$ is reversible w.r.t.\@ both $\mathcal{L}^{\rm EX}_{(\Lambda_1,{\bf c})}$ and $\mathcal{L}^{\rm EX}_{(\Lambda_2,{\bf c})}$, summation by parts yields
\begin{align*}
\nu_\alpha^{\Lambda_2}\left[ f\left(-\mathcal{L}^{\rm EX}_{(\Lambda_1,{\bf c})}f\right)\right] &= \frac{1}{2} \sum_{\substack{xy\in E(\Gamma)\\ x,y \in \Lambda_1}} c_{xy} \nu_\alpha^{\Lambda_2}\left[\left(\nabla_{xy} f\right)^2\right] \\ & \leq \frac{1}{2} \sum_{\substack{xy\in E(\Gamma)\\ x,y \in \Lambda_2}} c_{xy} \nu_\alpha^{\Lambda_2}\left[\left(\nabla_{xy} f\right)^2\right] 
=\nu_\alpha^{\Lambda_2} \left[ f \left( -\mathcal{L}^{\rm EX}_{(\Lambda_2,{\bf c})}f\right)\right].
\end{align*}
\end{proof} 

The one-block estimate hinges upon the following spectral estimate.


\begin{lemma}
\label{lem:1beigv}
For each $\kappa>0$, let $\lambda^{(1),\pm}_{N,j}(\kappa)$ be the largest eigenvalue of $\mathcal{T}_N \mathcal{L}^{\rm EX}_{(\Gamma_N, {\bf c})} \pm \kappa \mathcal{V}_N \Uone{j}(\pt,\cdot)$ with respect to the product Bernoulli measure $\nu_\alpha$. Then under Assumption \ref{ass:1}, we have
\begin{align}
\label{1beigv}
\limsup_{j\to\infty} \limsup_{N\to\infty} \frac{1}{\kappa \mathcal{V}_N} \lambda^{(1),\pm}_{N,j}(\kappa) \leq 0.
\end{align}
\end{lemma}
\begin{proof}
We use the variational characterization
\begin{equation}
\label{eq:eigvvar}
\frac{1}{\kappa \mathcal{V}_N}\lambda^{(1),\pm}_{N,j}(\kappa) = \sup_f \left\{ \nu_\alpha \left[\pm \Uone{j}(\pt,\cdot) f\right] - \frac{\mathcal{T}_N}{\kappa \mathcal{V}_N} \nu_\alpha \left[\sqrt{f} \left(-\mathcal{L}^{\rm EX}_{(\Gamma_N,{\bf c})} \sqrt{f}\right)\right]\right\},
\end{equation} 
where the supremum is taken over all densities $f$ with respect to $\nu_\alpha$.

First note that for each $j\in \mathbb{N}$, there exists $N_j$ such that $B(\pt, r_j) \subset B(o, r_N)$ for all $N\geq N_j$, so it suffices to consider all sufficiently large $N$ in what follows. By Lemma \ref{lem:DFEst1},
\begin{align*}
\nu_\alpha \left[\sqrt{f}\left(  -\mathcal{L}^{\rm EX}_{(\Lambda_j(\pt),{\bf c})} \sqrt{f} \right) \right] \leq \nu_\alpha \left[\sqrt{f} \left( -\mathcal{L}^{\rm EX}_{(\Gamma_N,{\bf c})} \sqrt{f}\right) \right].
\end{align*}
This leads to the upper estimate
\begin{align}
\label{eigvvar1}
\frac{1}{\kappa \mathcal{V}_N}\lambda^{(1),\pm}_{N,j}(\kappa) \leq \sup_f \left\{ \nu_\alpha \left[\pm \Uone{j}(\pt,\cdot) f\right] - \frac{\mathcal{T}_N}{\kappa \mathcal{V}_N} \nu_\alpha \left[\sqrt{f} \left(-\mathcal{L}^{\rm EX}_{(\Lambda_j(\pt),{\bf c})} \sqrt{f}\right)\right]\right\}.
\end{align} 
Observe that the variational functional on the RHS of \eqref{eigvvar1} depends only on $\{\eta(z): z\in B(\pt, r_j)\}$, which is a finite set by the assumption that the graph is locally finite. Since $\nu_\alpha$ is a product measure, it suffices to supremize the variational functional over densities $f$ relative to $\nu_\alpha^{\Lambda_j(\pt)}$:
\begin{align}
\label{1bspecineq}
\frac{1}{\kappa \mathcal{V}_N} \lambda^{(1),\pm}_{N,j}(\kappa) \leq \sup_f \left\{ \nu_\alpha^{\Lambda_j(x)} \left[ \pm\Uone{j}(\pt,\cdot) f\right] - \frac{\mathcal{T}_N}{\kappa \mathcal{V}_N} \nu_\alpha^{\Lambda_j(\pt)} \left[\sqrt{f}\left( -\mathcal{L}^{\rm EX}_{(\Lambda_j(\pt),{\bf c})} \sqrt{f}\right)\right]\right\}.
\end{align}

We now take the limit $N\to\infty$ on both sides of \eqref{1bspecineq}. Since the supremum runs over a compact set, we can interchange the limit and the supremum to find
\begin{align}
\label{1bsup}
\limsup_{N\to\infty} \frac{1}{\kappa \mathcal{V}_N} \lambda^{(1),\pm}_{N,j}(\kappa) & \leq
\sup_f \limsup_{N\to\infty} \left\{ \nu_\alpha^{\Lambda_j(\pt)} \left[ \pm\Uone{j}(\pt,\cdot) f\right] - \frac{\mathcal{T}_N}{\kappa \mathcal{V}_N} \nu_\alpha^{\Lambda_j(\pt)} \left[\sqrt{f}\left( -\mathcal{L}^{\rm EX}_{(\Lambda_j(\pt),{\bf c})} \sqrt{f}\right)\right]\right\} \\
\nonumber & = \sup_{f:~ \nu_\alpha^{\Lambda_j(\pt)}\left[\sqrt{f}\left(-\mathcal{L}^{\rm EX}_{(\Lambda_j(\pt), {\bf c})}\sqrt{f}\right)\right]=0} \left\{\nu_\alpha^{\Lambda_j(\pt)} \left[\pm \Uone{j}(\pt,\cdot) f\right] \right\}.
\end{align}
In going from the first line to the second line, we used Assumption \ref{ass:1} as well as the nonnegativity of the Dirichlet form,
\begin{align*} \nu_\alpha^{\Lambda_j(\pt)} \left[\sqrt{f}\left( -\mathcal{L}^{\rm EX}_{(\Lambda_j(\pt),{\bf c})} \sqrt{f}\right)\right] = \frac{1}{2} \sum_{\substack{zw\in E(\Gamma)\\ z,w\in \Lambda_j(\pt)}} c_{xy} \,\nu_\alpha^{\Lambda_j(\pt)}  \left[\left(\nabla_{zw} \sqrt f\right)^2 \right],
\end{align*}
thanks to the reversibility of $\mathcal{L}^{\rm EX}_{(\Lambda_j(\pt),{\bf c})}$ w.r.t.\@ $\nu_\alpha^{\Lambda_j(\pt)}$.

Recall also that the total particle number is conserved in the SEP. Since $\{0,1\}^{\Lambda_j(\pt)}$ is irreducible w.r.t.\@ $\mathcal{L}^{\rm EX}_{(\Lambda_j(\pt), {\bf c})}$, a probability density $f$ with zero Dirichlet energy,
\begin{align*}
\nu_\alpha^{\Lambda_j(\pt)}\left[\sqrt{f}\left(-\mathcal{L}^{\rm EX}_{(\Lambda_j(\pt), {\bf c})}\sqrt{f}\right)\right]=0,
\end{align*}
is constant on each hyperplane $\mathcal{X}_{\Lambda_j(\pt),k}$ with $k$ total number of particles, where 
\begin{align*}
\mathcal{X}_{\Lambda, k} := \left\{\eta\in \{0,1\}^\Lambda: \sum_{y\in \Lambda} \eta(y)=k\right\}.
\end{align*}
We let $\nu^\Lambda_{*,k} := \nu^\Lambda_{\alpha}\left(~\cdot~|\mathcal{X}_{\Lambda, k}\right)$, noting that this conditional measure is independent of $\alpha$, and that the random variables $\{\eta(z): z\in \Lambda\}$ are exchangeable under $\nu^{\Lambda}_{*,k}$. Thus the RHS of (\ref{1bsup}) equals
\begin{align}
\label{supU1}
\sup_{0\leq k\leq |\Lambda_j(\pt)|} \nu_{*,k}^{\Lambda_j(\pt)} \left[ \pm U^{(1)}_{j}(\pt,\eta)\right] &= \sup_{0\leq k\leq |\Lambda_j(\pt)|} \pm \left(\nu_{*,k}^{\Lambda_j(\pt)}\left[\phi_\pt\right] - \nu_{*,k}^{\Lambda_j(\pt)}\left[\Phi_\pt(\avg{\cdot}{\Lambda_j(\pt)})\right] \right)\\
\nonumber &= \sup_{0\leq k\leq |\Lambda_j(\pt)|} \pm \left(\nu_{*,k}^{\Lambda_j(\pt)}\left[\phi_\pt\right] - \nu_{k|\Lambda_j(\pt)|^{-1}}^{\Lambda_j(\pt)} \left[\phi_\pt\right] \right),
\end{align}
where we used the fact that $\avg{\eta}{\Lambda} = k|\Lambda|^{-1}$ for every $\eta \in \mathcal{X}_{\Lambda,k}$.

It remains to show that the limit of (\ref{supU1}) vanishes as $j \to\infty$, see the lemma below.
\end{proof}

\begin{lemma}[Equivalence of ensembles]
\label{lem:ensembles}
Let $(\Lambda_j)_{j\geq 1}$ be an increasing sequence of subsets in $V(\Gamma)$, with $|\Lambda_j| \to \infty$. Then
\begin{align}
\label{eq:ensembles}
\lim_{j\to\infty} \sup_{0 \leq k\leq |\Lambda_j|} \left(\nu_{*,k}^{\Lambda_j}[g]-\nu_{k|\Lambda_j|^{-1}}^{\Lambda_j}[g] \right)=0
\end{align}
 for every $g \in L^1(\{0,1\}^{V(\Gamma)}, \nu_\alpha)$.
\end{lemma}
\begin{proof}
This follows from the convergence in distribution of the hypergeometric to the binomial, $\nu^{\Lambda_j}_{*,k} \overset{d}{\to} \nu^{\Lambda_j}_{k|\Lambda_j|^{-1}}$ as $j\to\infty$.
\end{proof}
 
The time-dependent version of Lemma \ref{lem:1beigv} has $\lambda^{(1),\pm}_{N,j}(\kappa, t)$ being the largest eigenvalue of $\mathcal{T}_N \mathcal{L}^{\rm EX}_{\Gamma_N, c)}\pm \kappa \mathcal{V}_N U_j^{(1)}(\pt, \eta_t^N)$, for $t\in [0,T]$. By the same exact reasoning one shows \eqref{1beigv} with $\lambda^{(1),\pm}_{N,j}(\kappa, t)$ in place of $\lambda^{(1),\pm}_{N,j}(\kappa)$.

We now combine the spectral estimate, the Feynman-Kac formula, and the exponential Chebyshev's inequality to prove Theorem \ref{thm:1block}.

\begin{proof}[Proof of Theorem \ref{thm:1block}]
By the exponential Chebyshev's inequality, for any $\delta>0$ and $\kappa>0$ we have
\begin{align}
\label{ineq:U}
&~\quad\frac{1}{\mathcal{V}_N} \log \mathbb{P}^N_{\alpha}\left[\left|\int_0^T\, \Uone{j}(\pt,\eta^N_t)\,dt\right| >\delta \right] \\ 
\nonumber &\leq \frac{1}{\mathcal{V}_N} \log \left[e^{-\kappa\delta \mathcal{V}_N} \mathbb{E}^N_{\alpha}\left[\exp\left(\kappa \mathcal{V}_N\left|\int_0^T  \, \Uone{j}(\pt,\eta^N_t)\,dt\right| \right)\right]\right]\\
\nonumber &= \frac{1}{\mathcal{V}_N} \log \mathbb{E}^N_{\alpha} \left[\exp\left(\kappa \mathcal{V}_N \left| \int_0^T \, \Uone{j}(\pt,\eta_t^N)\,dt\right|\right)\right] - \kappa\delta.
\end{align}
We claim that the first term on the RHS of \eqref{ineq:U} vanishes in the limit $N\to\infty$ followed by $j\to\infty$. By the Feynman-Kac formula, 
\begin{align}
\label{FKeigv}
\log \mathbb{E}^N_{\alpha}\left[\exp\left(\pm \kappa \mathcal{V}_N \int_0^T\, \Uone{j}(\pt,\eta_t^N)\,dt\right)\right] \leq T \sup_{t\in [0,T]} \lambda^{(1),\pm}_{N,j}(\kappa,t),
\end{align}
where $\mathbb{E}^N_\alpha$ is the expectation corresponding to $\mathbb{P}^N_\alpha$, and $\lambda^{(1),\pm}_{N,j}(\kappa,t)$ is the largest eigenvalue of $\eta_t^N \mapsto \mathcal{T}_N \mathcal{L}^{\rm EX}_{(\Gamma_N,{\bf c})} \pm \kappa \mathcal{V}_N \Uone{j}(\pt,\eta_t^N)$ with respect to $\nu_\alpha$. Using the inequality $e^{|x|} \leq e^x + e^{-x}$, (\ref{FKeigv}), (\ref{px1x2}) and Lemma \ref{lem:1beigv}, we get
\begin{align*}
& ~\quad\limsup_{j\to\infty} \limsup_{N\to\infty} \frac{1}{\mathcal{V}_N}  \log \mathbb{E}^N_{\alpha} \left[\exp\left(\kappa \mathcal{V}_N \left| \int_0^T \, \Uone{j}(\pt,\eta_t^N)\,dt\right|\right)\right]\\
\nonumber & \leq  \kappa \limsup_{j\to\infty} \limsup_{N\to\infty}\frac{1}{\kappa \mathcal{V}_N}\left(\int_0^T \, \lambda^{(1),+}_{N,j}(\kappa,t) \,dt  +  \int_0^T \, \lambda^{(1),-}_{N,j}(\kappa,t) \,dt \right)  \\
\nonumber & \leq  \kappa \limsup_{j\to\infty} \limsup_{N\to\infty}\max\left(\frac{1}{\kappa  \mathcal{V}_N}\int_0^T \, \lambda^{(1),+}_{N,j}(\kappa,t) \,dt, \,  \frac{1}{\kappa  \mathcal{V}_N}  \int_0^T \, \lambda^{(1),-}_{N,j}(\kappa,t) \,dt \right)\\
\nonumber & =  \kappa \max\left(\limsup_{j\to\infty} \limsup_{N\to\infty}\frac{1}{\kappa  \mathcal{V}_N}\int_0^T \, \lambda^{(1),+}_{N,j}(\kappa,t) \,dt, \, \limsup_{j\to\infty} \limsup_{N\to\infty} \frac{1}{\kappa  \mathcal{V}_N}  \int_0^T \, \lambda^{(1),-}_{N,j}(\kappa,t) \,dt \right) ~\leq~ 0.
\end{align*}
Infer that
\begin{align*}
\limsup_{j\to\infty} \limsup_{N\to\infty} \frac{1}{\mathcal{V}_N} \log \mathbb{P}^N_{\alpha}\left[\left|\int_0^T\, \Uone{j}(\pt,\eta^N_t)\,dt\right| >\delta \right] \leq -\kappa \delta.
\end{align*}
Sending $\kappa\to\infty$ yields (\ref{1bp}).
\end{proof}

\section{Two-blocks estimate: Proof of Theorem \ref{thm:2block}} \label{sec:2blocks}

Let $N \in \mathbb{N}$, $\epsilon\in [0,1]$, and $x\in B(o,r_N)$. Let $\{\Lambda_j(x)\}_{j \geq 1}$ be a sequence of finite, connected subsets of $V(\Gamma)$ containing $x$, with $\Lambda_j(x) \subset B(x, r_j)$ and $|\Lambda_j(x)|\leq |\Lambda_{j+1}(x)|$ for each $j\geq 1$. For each $j\geq 1$, take any subset $\Lambda_j(y)$ of $B(x, r_{\epsilon N}) \setminus \Lambda_j(x)$ with $|\Lambda_j(y)| = |\Lambda_j(x)|$. (Here $y$ indexes an arbitrary vertex in $\Lambda_j(y)$.) From these we define
\begin{align}
\label{Utildetwo}
\Utildetwo{j}{x}{y}(\eta) := \avg{\eta}{\Lambda_j(x)} - \avg{\eta}{\Lambda_j(y)}, \quad \eta\in \{0,1\}^{V(\Gamma)}.
\end{align}
The ``two blocks'' in our setting refer to $\Lambda_j(x)$ and $\Lambda_j(y)$, which are disjoint by construction. In what follows we use the shorthand $\Lambda^{(2)}(j,x,y) = \Lambda_j(x) \cup \Lambda_j(y)$. Also, since we will take the limit $N\to\infty$ before all other limits, it suffices to consider all $N\geq N_j$ for an appropriate cutoff $N_j \in \mathbb{N}$.

\begin{remark}
For the purpose of proving the one-block estimate, we took $\Lambda_j(x)$ to be connected. However, we do not insist that $\Lambda_j(y)$ be connected. In the proof to follow, we will identify a nearly exact cover of $B(x, r_{\epsilon N}) \setminus \Lambda_j(x)$ by $\{\Lambda_j(y): y \in B(x,r_{\epsilon N})\setminus \Lambda_j(x)\}$ without insisting on the $\Lambda_j(y)$ to be connected.
\end{remark}

\subsection{Spectral estimate}

Our first task is to establish a spectral estimate involving $\Utildetwo{j}{x}{y}(\eta)$, analogous to Lemma \ref{lem:1beigv}. A notable difference here is that we need to capture the energy of transporting a particle from $\Lambda_j(x)$ to $\Lambda_j(y)$ via the so-called \emph{moving particle lemma}. This was proved for any finite weighted graph in \cite{ChenMPL}, based on the octopus inequality of \cite{CLR09}.

\begin{proposition}[Moving particle lemma \cite{ChenMPL}*{Theorem 1}]
\label{prop:MPL}
Let $(G,{\bf c})$ be a finite weighted graph. For every $\alpha\in [0,1]$, $x,y \in V(G)$, and $f: \{0,1\}^{V(G)} \to\mathbb{R}$, 
\begin{align}
\label{ineq:dirichletex}
\nu_{\alpha}[f(-\nabla_{xy} f)] \leq R^{(G,{\bf c})}_{\rm eff}(x,y) \, \nu_\alpha \left[f\left(-\mathcal{L}^{\rm EX}_{(G,{\bf c})} f\right) \right],
\end{align}
where
\[
R^{(G,{\bf c})}_{\rm eff}(x,y) = \left(\inf\left\{\mathcal{E}^{\rm el}_{(G,{\bf c})}(f) ~\bigg|~ f:V(G)\to\mathbb{R},~ f(x)=1,~f(y)=0\right\}\right)^{-1}
\]
 is the effective resistance between $x$ and $y$ in $(G, {\bf c})$.
\end{proposition}

Here is the key spectral estimate.

\begin{lemma}
\label{lem:2beigv}
For each $\kappa>0$, let $\lambda^{(2),\pm}_{N,j,x,y}(\kappa)$ be the largest eigenvalue of $\mathcal{T}_N \mathcal{L}^{\rm EX}_{(\Gamma_N,{\bf c})} \pm \kappa \mathcal{V}_N \Utildetwo{j}{x}{y}(\cdot)$ with respect to $\nu_\alpha$. Then under Assumption \ref{ass:2}, we have
\begin{align}
\label{2beigv}
\limsup_{j\to\infty} \limsup_{\epsilon\downarrow 0} \limsup_{N\to\infty} \frac{1}{\kappa \mathcal{V}_N} \lambda^{(2),\pm}_{N,j,x,y}(\kappa) \leq 0.
\end{align}
\end{lemma}

\begin{proof}
We start with the variational characterization
\begin{equation}
\frac{1}{\kappa \mathcal{V}_N}\lambda^{(2),\pm}_{N,j,x,y}(\kappa)=\sup_f \left\{\nu_\alpha \left[\pm\Utildetwo{j}{x}{y}(\cdot) f\right] -  \frac{\mathcal{T}_N}{\kappa \mathcal{V}_N} \nu_\alpha \left[\sqrt{f}\left( -\mathcal{L}^{\rm EX}_{(\Gamma_N,{\bf c})} \sqrt{f}\right) \right] \right\},
\end{equation}
where the supremum is taken over all densities $f$ relative to $\nu_\alpha$. Below we focus on the ``$+$'' case; the ``$-$'' case is proved in the same way.

Observe that $\Utildetwo{j}{x}{y}(\eta)$ depends only on $\{\eta(z): z\in \Lambda^{(2)}(j,x,y)\}$. Following the proof of Lemma \ref{lem:1beigv}, one would like to replace $\mathcal{L}^{\rm EX}_{(\Gamma_N,{\bf c})}$ in the variational functional by a generator which effectively acts on $\{0,1\}^{\Lambda^{(2)}(j,x,y)}$. This generator is of the form
\begin{align}
\mathcal{L}^{(2)}_{j,x,y} := \mathcal{L}^{\rm EX}_{(\Lambda_j(x), {\bf c})} + \mathcal{L}^{\rm EX}_{(\Lambda_j(y), {\bf c})} +\sum_{i=0}^{B-1} \nabla_{z_i z_{i+1}},
\end{align}
where $z_0 \in \Lambda_j(x)$, $B$ is the number of connected components of $\Lambda_j(y)$, and for each $i\in \{1,2,\cdots, B\}$, $z_i$ is a vertex belonging to the $i$th connected component of $\Lambda_j(y)$. (Our result is not sensitive to the particular choices of the $z_i$. Also note that $B\leq |\Lambda_j(y)| = |\Lambda_j(x)|$.) Observe that the first two terms in the generator $\mathcal{L}^{(2)}_{j,x,y}$ generates the exclusion processes in the individual blocks. In order to make the state space irreducible w.r.t.\@ the Markov chain to ensure ergodicity, we need to add the third term to facilitate particle jumps between the connected components of $\Lambda^{(2)}(j,x,y)$. 

By Lemma \ref{lem:DFEst1} we have
\begin{align*}
\nu_\alpha \left[\sqrt{f}\left( -\left(\mathcal{L}^{\rm EX}_{(\Lambda_j(x),{\bf c})} + \mathcal{L}^{\rm EX}_{(\Lambda_j(y), {\bf c})}\right)\sqrt{f}\right)\right] \leq  \nu_\alpha \left[ \sqrt{f}\left(-\mathcal{L}^{\rm EX}_{(\Gamma_N,{\bf c})} \sqrt{f} \right)\right].
\end{align*}
Meanwhile Proposition \ref{prop:MPL} (the moving particle lemma) yields
\begin{align}
\label{mpl2b1}
\nu_\alpha \left[ \sqrt{f}\left( -\nabla_{z_i z_{i+1}}\sqrt{f}\right)\right] \leq R^{(\Gamma_N,{\bf c})}_{\rm eff}(z_i,z_{i+1})\, \nu_\alpha \left[ \sqrt{f}\left( -\mathcal{L}^{\rm EX}_{(\Gamma_N,{\bf c})} \sqrt{f}\right)\right].
\end{align}
Combining the two inequalities together we get
\begin{align}
\label{ineqmoving}
\nu_\alpha \left[\sqrt{f}\left(-\mathcal{L}^{(2)}_{j,x,y}\sqrt{f}\right)\right] \leq \left(1+ \sum_{i=0}^{B-1} R_{\rm eff}^{(\Gamma_N, {\bf c})}(z_i, z_{i+1})\right) \nu_\alpha \left[ \sqrt{f} \left( -\mathcal{L}^{\rm EX}_{(\Gamma_N,{\bf c})}\sqrt{f}\right)\right].
\end{align}

Based on the inequality (\ref{ineqmoving}), the fact that $\Utildetwo{j}{x}{y}$ depends only on $\{\eta(z): z \in \Lambda^{(2)}(j,x,y)\}$, and that $\nu_\alpha$ is a product measure, it is enough to show that
\begin{align}
\label{2bsup}
\limsup_{j\to\infty} & \limsup_{\epsilon \downarrow 0} \limsup_{N\to\infty} \sup_f \left\{ \nu_\alpha^{\Lambda^{(2)}(j,x,y)}\left[ \Utildetwo{j}{x}{y}(\cdot) f\right] \right. \\ 
\nonumber & - \left.\frac{\mathcal{T}_N}{\kappa \mathcal{V}_N}\left(1+\sum_{i=0}^{B-1} R_{\rm eff}^{(\Gamma_N, {\bf c})}(z_i,z_{i+1})\right)^{-1} \nu_\alpha^{\Lambda^{(2)}(j,x,y)} \left[\sqrt{f}\left( -\mathcal{L}^{(2)}_{j,x,y} \sqrt{f}\right)\right]\right\}\leq 0,
\end{align}
where the supremum is taken over all densities with respect to $\nu_\alpha^{\Lambda^{(2)}(j,x,y)}$. We note that $\mathcal{L}^{(2)}_{j,x,y}$ is reversible w.r.t\@ $\nu_\alpha^{\Lambda^{(2)}(j,x,y)}$, which means that the Dirichlet form $ \nu_\alpha^{\Lambda^{(2)}(j,x,y)} \left[\sqrt{f}\left( -\mathcal{L}^{(2)}_{j,x,y} \sqrt{f}\right)\right]$ is nonnegative.

Since the supremum is over a compact set, we can interchange the limit and the supremum, and obtain
\begin{align}
\label{2bsup3}
& \quad\limsup_{\epsilon\downarrow 0}\limsup_{N\to\infty} \sup_f \bigg\{ \nu_\alpha^{\Lambda^{(2)}(j,x,y)}\left[ \Utildetwo{j}{x}{y}(\cdot) f\right] \\ 
\nonumber & \qquad \quad - \frac{\mathcal{T}_N}{\kappa \mathcal{V}_N}\left(1+\sum_{i=0}^{B-1} R_{\rm eff}^{(\Gamma_N, {\bf c})}(z_i,z_{i+1})\right)^{-1} \nu_\alpha^{\Lambda^{(2)}(j,x,y)} \left[\sqrt{f}\left( -\mathcal{L}^{(2)}_{j,x,y} \sqrt{f}\right)\right]\bigg\} \\
\nonumber &= \sup_{f:~ \nu_\alpha^{\Lambda^{(2)}(j,x,y)} \left[\sqrt{f}\left( -\mathcal{L}^{(2)}_{j,x,y} \sqrt{f}\right)\right]=0} \left\{ \nu_\alpha^{\Lambda^{(2)}(j,x,y)}\left[ \Utildetwo{j}{x}{y}(\cdot) f\right]  \right\}.
\end{align}
In establishing the equality in \eqref{2bsup3}, we used the nonnegativity of the Dirichlet form, as well as the asymptotic result
\begin{align}
\label{2basymp}
\liminf_{\epsilon \downarrow 0} \liminf_{N\to\infty}\frac{\mathcal{T}_N}{\mathcal{V}_N}\left(1+\sum_{i=0}^{B-1} R_{\rm eff}^{(\Gamma_N, {\bf c})}(z_i,z_{i+1})\right)^{-1}=\infty,
\end{align}
which we justify now. 

Let
\begin{align*}
R^{(\Gamma_N,{\bf c})}_{\max} := \max_{i\in \{0,1,\cdots, B-1\}} R_{\rm eff}^{(\Gamma_N, {\bf c})}(z_i, z_{i+1}),
\end{align*}
which is independent of $N$ and $\epsilon$, since $B\leq |\Lambda_j(x)|<\infty$ is independent of $N$ and $\epsilon$.
Now
\[
1+ \sum_{i=0}^{B-1} R_{\rm eff}^{(\Gamma_N,{\bf c})}(z_i, z_{i+1}) \leq (B+1) \left(1\vee R^{(\Gamma_N,{\bf c})}_{\max}\right),
\]
so the LHS of \eqref{2basymp} can be bounded below by
\[
\liminf_{\epsilon\downarrow 0} \liminf_{N\to\infty} \frac{\mathcal{T}_N}{\mathcal{V}_N} \frac{1}{B+1} \frac{1}{1\vee R_{\max}^{(\Gamma_N,{\bf c})}},
\]
which diverges by either Assumption \ref{ass:1} or \ref{ass:2}.


The remaining argument now mimics the end of the proof of Lemma \ref{lem:1beigv}. Since $\Lambda^{(2)}(j,x,y)$ is irreducible w.r.t.\@ $\mathcal{L}^{(2)}_{j,x,y}$, any probability density $f$ w.r.t.\@ $\nu_\alpha^{\Lambda^{(2)}(j,x,y)}$ which satisfies 
\begin{align*}
\nu_\alpha^{\Lambda^{(2)}(j,x,y)}\left[\sqrt{f}\left(-\mathcal{L}^{(2)}_{j,x,y} \sqrt{f}\right)\right]=0
\end{align*}
is constant on each hyperplane $\mathcal{X}_{\Lambda^{(2)}(j,x,y),k}$ with $k$ total number of particles. Let $\nu^{\Lambda^{(2)}(j,x,y)}_{*,k} = \nu^{\Lambda^{(2)}(j,x,y)}_\alpha\left(~\cdot~|\mathcal{X}_{\Lambda^{(2)}(j,x,y),k}\right)$, and note that $\{\eta(z): z\in \Lambda^{(2)}(j,x,y)\}$ is exchangeable under $\nu^{\Lambda^{(2)}(j,x,y)}_{*,k}$. Thus \eqref{2bsup3} equals
\begin{align}
\sup_{0\leq k\leq |\Lambda^{(2)}(j,x,y)|} \nu^{\Lambda^{(2)}(j,x,y)}_{*,k} \left[\Utildetwo{j}{x}{y}(\cdot)\right] = \sup_{0\leq k \leq |\Lambda^{(2)}(j,x,y)|} \nu^{\Lambda^{(2)}(j,x,y)}_{*,k}  \left[\avg{\eta}{\Lambda_j(x)}- \avg{\eta}{\Lambda_j(y)}\right].
\end{align}
We claim that this vanishes as $j\to\infty$ by the local central limit theorem below.
\end{proof}

\begin{lemma}
\label{lem:U2vanish}
Let $(\Lambda_j^1)_{j\geq 1}$ and $(\Lambda_j^2)_{j\geq 1}$ be two increasing sequences of denumerable sets such that:
\begin{itemize}
\item $\Lambda_j^1 \cap \Lambda_j^2 =\emptyset$ for each $j\geq 1$.
\item $|\Lambda_j^1| = |\Lambda_j^2|$ for each $j\geq 1$.
\item $\lim_{j\to\infty} |\Lambda_j^1| =\infty$.
\end{itemize}
Denote $\Lambda^{(2)}_j = \Lambda_j^1 \cup \Lambda_j^2$. Let $\nu^{\Lambda^{(2)}_j}_{*,k}$ denote the exchangeable probability measure on 
$$\mathcal{X}_{\Lambda^{(2)}_j,k}:=\left\{\eta\in \{0,1\}^{\Lambda^{(2)}_j} : \sum_{y\in \Lambda^{(2)}_j} \eta(y)=k\right\}.$$
 Then
\begin{align}
\label{eq:U2vanish}
\lim_{j\to\infty} \sup_{0\leq k\leq |\Lambda^{(2)}_j|} \nu^{\Lambda^{(2)}_j}_{*,k} \left[ \left|\avg{\eta}{\Lambda_j^1} -\avg{\eta}{\Lambda_j^2}\right|\right] =0.
\end{align}
\end{lemma}
\begin{proof}
While a more powerful proof exists (\emph{cf.\@} \cite{Petrov}*{Theorem VII.12}), here we give a proof using elementary discrete probability. 

First assume that $k \leq |\Lambda_j^1| = \frac{1}{2}|\Lambda_j^{(2)}|$. The expectation in (\ref{eq:U2vanish}) boils down to a sum
\begin{equation}
\label{eq:hyper}
\sum_{m=0}^k \frac{|k-2m|}{|\Lambda_j^1|} \frac{\displaystyle {k\choose m}{2|\Lambda_j^1|-k \choose |\Lambda_j^1|-m}}{\displaystyle {2|\Lambda_j^1| \choose |\Lambda_j^1|}}.
\end{equation} 
The second term in the summand represents the $\nu^{\Lambda_j^{(2)}}_{*,k}$-probability assigned to the space of configurations having $m$ particles in $\Lambda_j^1$ and $(k-m)$ particles in $\Lambda_j^2$. We recognize this as the hypergeometric distribution, arising from sampling without displacement. By the Cauchy-Schwarz inequality, (\ref{eq:hyper}) is bounded above by
\begin{align}
\label{eq:hyper2}
\frac{1}{|\Lambda_j^1|} \left[ \sum_{m=0}^k |k-2m|^2 \frac{\displaystyle {k\choose m}{2|\Lambda_j^1|-k \choose |\Lambda_j^1|-m}}{\displaystyle {2|\Lambda_j^1| \choose |\Lambda_j^1|}}\right]^{1/2} =\frac{2}{|\Lambda_j^1|} (E|Y-E[Y]|^2)^{1/2}= \frac{2}{|\Lambda_j^1|} [{\rm Var}(Y)]^{1/2},
\end{align}
where $Y$ is a hypergeometric random variable with parameters $(2|\Lambda_j^1|, |\Lambda_j^1|, k)$, and $E$ and ${\rm Var}$ denote the corresponding expectation and variance. It is a standard result (see \emph{e.g.} \cite{Ross}*{\S4.8.3}) that $E[Y]=|\Lambda_j^1| \frac{k}{2|\Lambda_j^1|} = \frac{k}{2}$, and
\begin{align*}
{\rm Var}(Y) = |\Lambda_j^1| \frac{k}{2|\Lambda_j^1|} \left(1-\frac{k}{2|\Lambda_j^1|}\right)\left( \frac{2|\Lambda_j^1|-|\Lambda_j^1|}{2|\Lambda_j^1|-1}\right) = \frac{k}{2} \left(1-\frac{k}{2|\Lambda_j^1|}\right) \left(\frac{|\Lambda_j^1|}{2|\Lambda_j^1|-1}\right),
\end{align*} 
which attains a maximum $\frac{|\Lambda_j^1|}{4}\left(\frac{|\Lambda_j^1|}{2|\Lambda_j^1|-1}\right)$ at $k=|\Lambda_j^1|$. Therefore for all $k\leq |\Lambda_j^1|$, (\ref{eq:hyper2}) is bounded above by
\begin{align*}
\frac{2}{|\Lambda_j^1|} \left[\frac{|\Lambda_j^1|}{4} \left(\frac{|\Lambda_j^1|}{2|\Lambda_j^1|-1}\right)\right]^{1/2} = \frac{1}{|\Lambda_j^1|^{1/2}}\left(\frac{|\Lambda_j^1|}{2|\Lambda_j^1|-1}\right)^{1/2},
\end{align*}
which vanishes in the limit $j\to\infty$.
The same conclusion holds for the case $|\Lambda_j^1| < k \leq 2|\Lambda_j^1|$ by appealing to the particle-hole symmetry (particles under the measure $\nu_{*,k}^{\Lambda_j^{(2)}}$ are identified with holes under the measure $\nu_{*,|\Lambda_j^{(2)}|-k}^{\Lambda_j^{(2)}}$).
\end{proof}

\begin{remark}
The limit statement \eqref{eq:U2vanish} continues to hold even if the cardinalities $|\Lambda_j^1|$ and $|\Lambda_j^2|$ are unequal, \emph{i.e.,} in the hypothesis of Lemma \ref{lem:U2vanish}, one drops the second condition, and replaces the third condition with $\lim_{j\to\infty} |\Lambda_j^i|=\infty$, $i\in \{1,2\}$.
\end{remark}

\subsection{Partitioning and averaging}

Having proved the spectral estimate, Lemma \ref{lem:2beigv}, our next task is to connect $\Utildetwo{j}{\pt}{y}(\eta)$ to $\Utwo{N}{\epsilon}{j}(\pt,\eta)$. We begin with a general lemma concerning averages.

\begin{lemma}
\label{lem:avg}
Let $\Lambda$ be a finite set which admits the partition
\begin{align*}
\Lambda = \left(\bigcup_{i=1}^L \Lambda_i\right) \cup \left(\bigcup_{\ell=1}^D \mathfrak{T}_\ell\right).
\end{align*}
Then for every function $g: \Lambda\to\mathbb{R}$,
\begin{align*}
\avg{g}{\Lambda_1} - \avg{g}{\Lambda} &=  \sum_{i=2}^L \frac{1}{2}\left(\frac{1}{L} + \frac{|\Lambda_i|}{|\Lambda|}\right) \left(\avg{g}{\Lambda_1} - \avg{g}{\Lambda_i}\right) \\
\nonumber & \quad + \sum_{i=1}^L \frac{1}{2}\left(\frac{1}{L} - \frac{|\Lambda_i|}{|\Lambda|}\right) \left(\avg{g}{\Lambda_1} + \avg{g}{\Lambda_i}\right) \\
\nonumber & \quad - \sum_{\ell=1}^D \frac{|\mathfrak{T}_\ell|}{|\Lambda|} \avg{g}{\mathfrak{T}_\ell}.
\end{align*}
\end{lemma}

\begin{proof}
We have the easy identity
\begin{align*}
\avg{g}{\Lambda} = \sum_{i=1}^L \frac{|\Lambda_i|}{|\Lambda|} \avg{g}{\Lambda_i} + \sum_{\ell=1}^D \frac{|\mathfrak{T}_\ell|}{|\Lambda|} \avg{g}{\mathfrak{T}_\ell}.
\end{align*}
Therefore
\begin{align}
\nonumber \avg{g}{\Lambda_1} - \avg{g}{\Lambda} &= \avg{g}{\Lambda_1} - \sum_{i=1}^L \frac{|\Lambda_i|}{|\Lambda|} \avg{g}{\Lambda_i}- \sum_{\ell=1}^D \frac{|\mathfrak{T}_\ell|}{|\Lambda|} \avg{g}{\mathfrak{T}_\ell} \\
\label{avgdec} &= \sum_{i=1}^L \left(\frac{1}{L} \avg{g}{\Lambda_1} - \frac{|\Lambda_i|}{|\Lambda|} \avg{g}{\Lambda_i}\right) - \sum_{\ell=1}^D \frac{|\mathfrak{T}_\ell|}{|\Lambda|} \avg{g}{\mathfrak{T}_\ell}.
\end{align}
Note that the first term on the RHS of \eqref{avgdec} equals
\begin{align*}
\sum_{i=1}^L \frac{1}{2}\left(\frac{1}{L} + \frac{|\Lambda_i|}{|\Lambda|}\right) \left( \avg{g}{\Lambda_1} - \avg{g}{\Lambda_i} \right) + \sum_{i=1}^L \frac{1}{2}\left(\frac{1}{L} - \frac{|\Lambda_i|}{|\Lambda|}\right) \left( \avg{g}{\Lambda_1} + \avg{g}{\Lambda_i} \right).
\end{align*}
The index $i$ in the first sum in effect ranges from $2$ to $L$. The lemma follows.
\end{proof}

We wish to apply this lemma to the setting where $\Lambda_1 = \Lambda_j(\pt)$ and $\Lambda = B(\pt, r_{\epsilon N})$. To do this, we provide a partition of $B(\pt,r_{\epsilon N})$ consisting of many ``blocks'' and one ``tail.'' To be precise, a subset $\Lambda \subset B(\pt,r_{\epsilon N}) \setminus \Lambda_j(\pt)$ is a ''block'' if $|\Lambda| = |\Lambda_j(\pt)|$, and a ``tail'' if $|\Lambda|<|\Lambda_j(\pt)|$. No assumption on the connectedness of the blocks is necessary.

\begin{remark}
The preceding paragraph may remind the reader of the \emph{graph partitioning problem} in theoretical computer science, namely, finding an optimal partition (which minimizes some predefined cost function) of a graph on $n$ vertices into $k$ subgraphs, each containing at most $(1+\epsilon)\lfloor \frac{n}{k}\rfloor$ vertices. Given the assumptions imposed at the beginning of this paper, we do not need to exploit the full complexity of the graph partitioning problem; rather, the separation of microscopic and macroscopic scales ($j$ vs.\@ $\epsilon N$) is enough for our argument to go through. See \eqref{2bsup3}, \eqref{2basymp}, and the ensuing arguments for the key estimate. That said, it is possible to improve the estimates via optimizing over all possible graph partitions.
\end{remark}

\begin{proposition}
\label{prop:goodavg}
Consider the partition
\begin{align*}
B(\pt,r_{\epsilon N}) = \Lambda_j(\pt) \cup \left(\bigcup_{i=2}^{L_N} \Lambda_j(y_i) \right) \cup \mathfrak{T},
\end{align*}
where
\begin{align*}
L_N = L(\pt,\Lambda_j(\pt), N, \epsilon) = \left\lfloor \frac{|B(\pt,r_{\epsilon N})|}{|\Lambda_j(\pt)|} \right\rfloor,
\end{align*}
\begin{align*}
|\Lambda_j(y_i)| = |\Lambda_j(\pt)|, \quad i\in\{2,3,\cdots, L_N\}.
\end{align*}
Then
\begin{align*}
\left|\Utwo{N}{\epsilon}{j}(\pt,\eta)\right|:=\left|\avg{\eta}{\Lambda_j(\pt)} - \avg{\eta}{B(\pt, r_{\epsilon N})}\right| \leq \sup_{i\in \{2,3,\cdots, L_N\}} \left|\avg{\eta}{\Lambda_j(\pt)} - \avg{\eta}{\Lambda_j(y_i)}\right| + o_N(1)
\end{align*}
as $N\to\infty$.
\end{proposition}

\begin{proof}
Using Lemma \ref{lem:avg} and the triangle inequality we find
\begin{align*}
\left|\avg{\eta}{\Lambda_j(\pt)} - \avg{\eta}{B(\pt, r_{\epsilon N})}\right| \leq I_1 + I_2 + I_3,
\end{align*}
where
\begin{align*}
I_1 &= \sup_{i\in \{2,3,\cdots, L_N\}} \left| \avg{\eta}{\Lambda_j(\pt)} - \avg{\eta}{\Lambda_j(y_i)}\right| \cdot \sum_{i=2}^{L_N} \frac{1}{2} \left(\frac{1}{L_N} + \frac{|\Lambda_j(\pt)|}{L_N |\Lambda_j(\pt)| + |\mathfrak{T}|} \right),\\
I_2 &= \sum_{i=1}^{L_N} \frac{1}{2} \left|\frac{1}{L_N} - \frac{|\Lambda_j(\pt)|}{L_N|\Lambda_j(\pt)| + |\mathfrak{T}|} \right| \left|\avg{\eta}{\Lambda_j(\pt)} + \avg{\eta}{\Lambda_j(y_i)}\right|,\\
I_3 &= \frac{|\mathfrak{T}|}{L_N |\Lambda_j(\pt)| + |\mathfrak{T}|} |\avg{\eta}{\mathfrak{T}}|.
\end{align*}
Observe that the sum in $I_1$ is $\leq 1$. For $I_2$ and $I_3$, we can replace each instance of the average density $\avg{\eta}{\Lambda}$ by the upper bound $1$, resulting in the estimates
\begin{align*}
I_2 &\leq \sum_{i=1}^{L_N} \left|\frac{1}{L_N} - \frac{|\Lambda_j(\pt)|}{L_N|\Lambda_j(\pt)|+|\mathfrak{T}|}\right| = \left( 1+ L_N\frac{|\Lambda_j(\pt)|}{|\mathfrak{T}|}\right)^{-1}, \\
I_3 &\leq \frac{|\mathfrak{T}|}{L_N |\Lambda_j(\pt)|+ |\mathfrak{T}|} \leq \frac{|\mathfrak{T}|}{L_N |\Lambda_j(\pt)|}. 
\end{align*}
Since $|\mathfrak{T}| < |\Lambda_j(\pt)|$ by construction, we have $I_2 + I_3 = O\left(\frac{1}{L_N}\right)= o_N(1)$.
\end{proof}

We now have all the ingredients to prove the two-blocks estimate.

\begin{proof}[Proof of Theorem \ref{thm:2block}]
Recall that $\Phi_\pt: [0,1] \to \mathbb{R}$ is Lipschitz: there exists a positive constant $C_\Phi$ such that
\begin{align*}
\left|\Phi_\pt(\rho_1)- \Phi_\pt(\rho_2)\right| \leq C_\Phi |\rho_1-\rho_2|, \quad \rho_1,\rho_2 \in [0,1].
\end{align*}
Using this and Proposition \ref{prop:goodavg}, we obtain the estimate
\begin{align*}
\left|\int_0^T \, \Utwo{N}{\epsilon}{j}(\pt,\eta^N_t)\,dt\right|
& \leq  C_\Phi \,\int_0^T \,\left| \avg{\eta^N_t}{\Lambda_j(\pt)} - \avg{\eta^N_t}{B(\pt, r_{\epsilon N})}\right| \,dt\\
&\leq C_\Phi \left(\int_0^T \, \sup_{i\in \{2,3,\cdots, L_N\}}  \left|\Utildetwo{j}{\pt}{y_i}(\eta_t^N)\right|\,dt+T \cdot o_N(1)\right).
\end{align*}
Thus it suffices to prove that for each $i\in \{2,3,\cdots, L_N\}$, each $T>0$ and each $\delta>0$,
\begin{align*}
\limsup_{j\to\infty} \limsup_{\epsilon\downarrow 0} \limsup_{N\to\infty} \frac{1}{\mathcal{V}_N}\log \mathbb{P}^N_\alpha\left[ \left|\int_0^T\Utildetwo{j}{\pt}{y_i}(\eta^N_t)  \,dt\right|>\delta\right] = -\infty.
\end{align*}
This is proved using the spectral estimate, Lemma \ref{lem:2beigv}, and the exponential Chebyshev's inequality, in the same fashion as in the proof of Proposition \ref{thm:1block}.
\end{proof}

\section{Local ergodicity in the boundary-driven exclusion process: Proof of Theorem \ref{thm:localergodicboundary}} \label{sec:boundary}

\subsection{Dirichlet problem for the one-site marginal density} \label{sec:dirichlet}

As mentioned in \S\ref{sec:mainboundary}, the invariant measure $\mu_{\rm inv}$ of the boundary-driven exclusion process is not a product measure. However, it is possible to compute the one-site marginals of $\mu_{\rm inv}$, denoted $\rho(x) := E_{\mu_{\rm inv}}[\eta(x)]$, $x\in V(G)$. Indeed, using the fact that $\mu_{\rm inv}$ is invariant, we see that $E_{\mu_{\rm inv}}[(\mathcal{L}^{\rm bEX}_{(G,{\bf c})} \eta)(x)]=0$ for all $x\in V(G)$. Expanding this equality yields the following linear system:
\begin{align}
\label{ODErho}
\left\{
\begin{array}{ll}\displaystyle \sum_{y\in V(G)} c_{xy} [\rho(x) - \rho(y)]=0 &  \text{for all}~ x\in V(G)\setminus \partial V, \\ \displaystyle
\sum_{y\in V(G)\setminus \partial V} c_{ay} [\rho(a) - \rho(y)] = \lambda_+(a) [1-\rho(a)]  - \lambda_-(a) \rho(a) & \text{for all}~a \in \partial V.
   \end{array} \right.
\end{align}
Observe that $\rho$ is harmonic w.r.t.\@ the graph Laplacian on $V(G)\setminus \partial V$, and satisfies a Robin boundary condition on $\partial V$---the LHS of the second equation is $-i_a(\rho)$, the normal derivative of $\rho$ at $a\in \partial V$, equivalently, the flow of $\rho$ into $a\in \partial V$. 

We may recast \eqref{ODErho} into a pure Dirichlet problem.

\begin{proposition}[Dirichlet-to-Neumann map for $\rho$]
\label{prop:DtoN}
The unique solution $\rho$ of (\ref{ODErho}) is also the unique solution of the Dirichlet problem
\begin{align}
\label{ODErho2}
\left\{
\begin{array}{ll}\displaystyle \sum_{y\in V(G)} c_{xy} [\rho(x) - \rho(y)]=0 &  \text{for all}~ x\in V(G)\setminus \partial V,\\
\displaystyle \rho(a)\left[c_a+\lambda_+(a)+\lambda_-(a) \right] - \sum_{b\in \partial V} \hat{c}_{a,b} \rho(b) = \lambda_+(a) & \text{for all}~a \in \partial V,
   \end{array} \right.
\end{align}
where $\hat{c}_{a,b} := \sum_{y\in V(G)} c_{ay} h^b(y)$ for $a,b\in \partial V$, and $h^b$ is the unique solution of the Dirichlet problem
\begin{align}
\label{ODErho3}
\left\{
\begin{array}{ll}\displaystyle \sum_{y\in V(G)} c_{xy} [h^b(x) - h^b(y)]=0 &  \text{for all}~ x\in V(G)\setminus\partial V,\\
h^b(a) = \delta_b(a) & \text{for all}~a \in \partial V.
   \end{array} \right.
\end{align}
\end{proposition}


\begin{remark}
\label{rem:trace}
By the strong maximum principle, $\hat{c}_{a,b}$ is positive for each $a,b\in \partial V$. Define $\widetilde{p}(a,b) = \frac{\hat{c}_{a,b}}{c_a}$. We claim that $\widetilde{p}(\cdot,\cdot)$ is a transition probability, \emph{i.e.,} $c_a = \sum_{b\in \partial V} \hat{c}_{a,b}$. To see this, observe that $\sum_{b\in \partial V} h^b$ is the unique harmonic function with boundary values identically $1$, so $\sum_{b\in \partial V} h^b(y) =1$ for all $y\in V$. It follows that 
\begin{align*}
\sum_{b\in \partial V} \widetilde{p}(a,b) = \sum_{b\in \partial V} \frac{\hat{c}_{a,b}}{c_a} = \frac{1}{c_a} \sum_{b\in\partial V} \sum_{y\in V(G)} c_{ay} h^b(y) = \frac{1}{c_a} \sum_{y\in V(G)} c_{ay}= 1.
\end{align*}

The Markov chain on $\partial V$ with transition probability $\widetilde{p}(\cdot,\cdot)$ is the \emph{trace} of the symmetric random walk process on $(G,{\bf c})$ to $\partial V$, which has an associated Dirichlet energy
\begin{align*}
{\rm Tr}_{\partial V} \mathcal{E}(g,g) := \frac{1}{2} \sum_{a\in \partial V} \sum_{b\in \partial V} \hat{c}_{a,b} [g(a)-g(b)]^2, \quad g: \partial V\to\mathbb{R}.
\end{align*}
It is proved in \cite{Ngasket}*{Theorem A.10} that 
\begin{align*}
{\rm Tr}_{\partial V} \mathcal{E}(g,g) = \mathcal{E}(h_g,h_g) := \sum_{xy\in E(G)} c_{xy} [h_g(x)-h_g(y)]^2,
\end{align*}
where $h_g$ is the unique harmonic extension of $g$ to $V(G)$. In general, $\tilde{p}(a,a)$ may be nonzero, and $\hat{c}_{a,b} \neq \hat{c}_{b,a}$, which implies that the trace process need not be symmetric.

We can also express the electric flow into $a\in \partial V$ in terms of the trace process. Given $h: V(G)\to \mathbb{R}$ which is harmonic on $V(G)\setminus \partial V$, we have
\begin{align*}
i_h(a) &:=\sum_{y\in V(G)} c_{ay}[h(y)-h(a)] = \sum_{y\in V(G)} c_{ay} \left[\sum_{b\in \partial V} h(b) h^b(y)- h(a)\right]\\
\nonumber &= \sum_{b\in \partial V} \hat{c}_{a,b} h(b) - c_a h(a) = \sum_{b\in \partial V} \hat{c}_{a,b} [h(b)-h(a)].
\end{align*} 
\end{remark}

\begin{proof}[Proof of Proposition \ref{prop:DtoN}]
Throughout the proof we write $V^0 := V(G) \setminus \partial V$. Let $\ell(S) =\{f: S\to\mathbb{R}\}$, and for each $x,y \in V$, let $p(x,y)=\frac{c_{xy}}{c_x}$ denote the transition probability from $x$ to $y$. Then we can rewrite (\ref{ODErho}) in block matrix form
\begin{align}
\label{hblock}
\begin{bmatrix}\boldsymbol{I-P} & -\boldsymbol{A} \\ -\boldsymbol{B}  & \boldsymbol{I+\sigma} \end{bmatrix}
\begin{bmatrix} \left.\boldsymbol{\rho}\right|_{V^0} \\ \left.\boldsymbol{\rho}\right|_{\partial V} \end{bmatrix} = 
\begin{bmatrix} 0 \\ \boldsymbol{\Lambda}_+ \end{bmatrix},
\end{align}
where 
\begin{align*}
\boldsymbol P : \ell(V^0) \to \ell(V^0),& \quad  ({\boldsymbol P}f)(x) = \sum_{y\in V^0} p(x,y) f(y),\\
\boldsymbol A : \ell(\partial V) \to \ell(V^0),& \quad ({\boldsymbol A}f)(x) = \sum_{a\in \partial V} p(x,a) f(a),\\
\boldsymbol B : \ell(V^0) \to \ell(\partial V),& \quad ({\boldsymbol B}f)(a) = \sum_{x \in V^0} p(a,x) f(x) ,\\ 
\boldsymbol \sigma : \ell(\partial V) \to \ell(\partial V), & \quad  ({\boldsymbol \sigma}f)(a) = c_a^{-1} [\lambda_+(a) + \lambda_-(a)]f(a),
\end{align*}
and $\boldsymbol \rho$ (resp.\@ ${\boldsymbol \Lambda}_+$) denotes the column vector with entries $\rho(x),~x\in V$ (resp. $c_a^{-1} \lambda_+(a),~a\in \partial V$).
Likewise, (\ref{ODErho3}) can be rewritten in block matrix form as
\begin{align}
\label{hblock2}
\begin{bmatrix} \boldsymbol{I-P} & -\boldsymbol{A} \\ 0 & \boldsymbol{I} \end{bmatrix} \begin{bmatrix} \left.\boldsymbol{h}^b \right|_{V^0} \\ \left.\boldsymbol{h}^b \right|_{\partial V}\end{bmatrix} = \begin{bmatrix} 0 \\ \delta_b\end{bmatrix}
\end{align}
with the same $\boldsymbol P$ and $\boldsymbol A$ as above. In particular,
\begin{align}
\label{eq:I-P}
(\boldsymbol I - \boldsymbol P) \left.\boldsymbol h^b\right|_{V^0} = \boldsymbol A \left.\boldsymbol h^b\right|_{\partial V}.
\end{align}

We claim that $\boldsymbol I - \boldsymbol P$ is invertible. Observe that
$$
\begin{bmatrix} \boldsymbol P & \boldsymbol A \\ 0 & \boldsymbol I \end{bmatrix}
$$
is the stochastic matrix of a Markov chain on $V(G)$ which is absorbed on $\partial V$, \emph{i.e.,} $p(a,a)=1$ for all $a\in \partial V$. Since $\boldsymbol P$ is the submatrix restricted to the transient states $V^o$, by \cite{GrinsteadSnell}*{Theorems 11.3 \& 11.4}, $\boldsymbol I - \boldsymbol P$ is invertible. ($(\boldsymbol I-\boldsymbol P)^{-1}$ is called the \emph{fundamental matrix} in \cites{GrinsteadSnell, DoyleSnell}.)
As a result \eqref{eq:I-P} yields
\begin{align}
\left[(\boldsymbol I-\boldsymbol P)^{-1} \boldsymbol A\right] \left.\boldsymbol h^b\right|_{\partial V} = \left.\boldsymbol h^b\right|_{V^0}.
\end{align} 

Referring back to \eqref{hblock}, we wish to identify the Dirichlet boundary condition for $\rho$. This can be achieved by eliminating $\left.\boldsymbol \rho\right|_{V^0}$ from the system \eqref{hblock}, \emph{i.e.,} computing the Schur complement of the block matrix with respect to the $V^0\times V^0$ subblock, to obtain
\begin{align}
\label{rhoSchur}
\left[(\boldsymbol I + \boldsymbol \sigma)-\boldsymbol B (\boldsymbol I -\boldsymbol P)^{\boldsymbol -1}\boldsymbol A \right] \left. \boldsymbol \rho\right|_{\partial V} = {\boldsymbol\Lambda}_+ 
\end{align}
Using the linearity of the Dirichlet harmonic functions, we write $\rho(x) = \sum_{b\in \partial V} \rho(b) h^b(x)$ for every $x\in V$. We can thus express the two terms on the LHS of (\ref{rhoSchur}) in component form as follows:
\begin{align*}
\left((\boldsymbol I + \boldsymbol \sigma) \left.\boldsymbol \rho\right|_{\partial V}\right)(a) &= \left(1+c_a^{-1} (\lambda_+(a) + \lambda_-(a))\right) \rho(a),\\
\left(\left[\boldsymbol B(\boldsymbol I-\boldsymbol P)^{-1} \boldsymbol A\right] \left.\boldsymbol \rho\right|_{\partial V}\right)(a) &= \sum_{b\in \partial V} \rho(b) \left(\left[\boldsymbol B(\boldsymbol I-\boldsymbol P)^{-1} \boldsymbol A\right] \left.\boldsymbol h^b\right|_{\partial V} \right)(a)\\
\nonumber &=\sum_{b\in \partial V} \rho(b) \left(\boldsymbol B \left.\boldsymbol h^b\right|_{V^0}\right) (a) \\
\nonumber &=\sum_{b\in \partial V} \rho(b) \sum_{x\in V^0} p(a,x) h^b(x) = \sum_{b\in \partial V} \frac{\hat{c}_{a,b}}{c_a} \rho(b)
\end{align*}
for each $a\in \partial V$. Combining these yields the Dirichlet boundary condition in (\ref{ODErho2}).
\end{proof}

We now show that the one-site marginal $\rho$ of $\mu_{\rm inv}$ is bounded away from $0$ and $1$.

\begin{corollary}
\label{cor:boundaryrate}
Under Condition (\nameref{ass:boundaryrate}), we have 
$\frac{1}{1+\gamma} \leq \rho(x) \leq \frac{\gamma}{1+\gamma}$ for all $x\in V(\Gamma)$, where $\gamma\geq 1$ is as in Assumption \ref{ass:boundaryrate}.
\end{corollary}

\begin{proof}
It suffices to prove the bounds for all $x\in \partial V$. Since $\rho$ is harmonic on $V^0$, the bounds also hold for all $x\in V^0$ as a consequence of the maximum principle for harmonic functions.

We begin with the boundary condition in \eqref{ODErho2}, which can be rewritten as
\begin{align}
\label{P=I+Lambda}
\tilde{\bf P} \left.\boldsymbol \rho\right|_{\partial V} = \left({\bf I} + \boldsymbol\Lambda_{\Sigma}\right) \left.\boldsymbol \rho\right|_{\partial V} - \boldsymbol\Lambda_+,
\end{align}
where $\tilde{\bf P}$ is a $|\partial V|\times |\partial V|$ matrix with entries $\tilde{p}(a,b)$, $\boldsymbol\Lambda_\Sigma$ is a $|\partial V|\times |\partial V|$ diagonal matrix with diagonal entries $c_a^{-1} (\lambda_+(a) + \lambda_-(a))$, and $\boldsymbol\Lambda_+$ is a $|\partial V|$-dimensional column vector with entries $c_a^{-1} \lambda_+(a)$. 

Let $\gamma, \gamma' \in [1,\infty)$ be as in Condition (\nameref{ass:boundaryrate}). We prove the lower bound $\rho(a) \geq \frac{1}{1+\gamma}$ for all $a\in \partial V$. Suppose, on the contrary, that $\inf_{a\in\partial V} \rho(a) = \kappa :=\frac{1}{1+\gamma} - \delta$ for some $\delta>0$. Since $\tilde{\bf P}$ is a stochastic matrix,  each component of the vector $\tilde{\bf P} \left.\boldsymbol \rho\right|_{\partial V}$ is a convex combination of $\{\rho(a): a\in \partial V\}$, and thus is $\geq \kappa$. Hence
\[
\inf_{a\in \partial V} \tilde{\bf P} \left.\boldsymbol \rho\right|_{\partial V} (a) \geq \kappa. 
\]
Meanwhile, for every $\epsilon>0$ there exists $a^* \in \partial V$ such that $\rho(a^*) < \kappa +\epsilon$, so
\begin{align*}
\left(\left({\bf I} + \boldsymbol\Lambda_{\Sigma}\right) \left.\boldsymbol \rho\right|_{\partial V} - \boldsymbol\Lambda_+\right) (a^*) &= \left(1+ c_{a^*}^{-1} (\lambda_+(a^*) + \lambda_-(a^*))\right) \rho(a^*) - c_{a^*}^{-1} \lambda_+(a^*)\\
&< (\kappa+\epsilon) + c_{a^*}^{-1} \left((\kappa+\epsilon-1)\lambda_+(a^*) + (\kappa+\epsilon) \lambda_-(a^*) \right).
\end{align*}
Using the lower bounds in Condition (\nameref{ass:boundaryrate}) we see that the last line is less than
\begin{align*}
(\kappa+\epsilon) + \gamma' \left( (1+\gamma)(\kappa+\epsilon)-1\right) &< \kappa+\gamma'(1+\gamma)\epsilon -\gamma'(1+\gamma)\delta.
\end{align*}
Since $\epsilon>0$ is arbitrary we deduce that 
\[
\inf_{a\in \partial V}\left(\left({\bf I} + \boldsymbol\Lambda_{\Sigma}\right) \left.\boldsymbol \rho\right|_{\partial V} - \boldsymbol\Lambda_+\right) (a) \leq \kappa-\gamma'(1+\gamma)\delta < \kappa,
\]
which leads to a contradiction in light of the identity \eqref{P=I+Lambda}.

The proof of the upper bound $\rho(a) \leq \frac{\gamma}{1+\gamma}$ for all $a\in \partial V$ is very similar, and thus omitted.
\end{proof}

Following \cite{BDGJL03}, we introduce a ``mock'' product Bernoulli measure $\nu_\lambda$ on $\{0,1\}^{V(G)}$ whose marginal is $\rho$, \emph{i.e.,} $\nu_\lambda\{\eta: \eta(x)=1\} = \rho(x)$ for each $x\in V(G)$. Corollary \ref{cor:boundaryrate} implies that $\nu_\lambda$ is absolutely continuous w.r.t.\@ any constant-density product Bernoulli measure $\nu_\alpha$, $\alpha\in (0,1)$, on $\{0,1\}^{V(G)}$. As a result, we can transfer the preceding theorems (the moving particle lemma and the one-block and two-block estimates) from the measure $\nu_\alpha$ to $\nu_\lambda$, modulo an issue of reversibility of $\nu_\lambda$ w.r.t.\@ the generators $\mathcal{L}^{\rm EX}_{(G,{\bf c})}$ and $\mathcal{L}^{b}_{\partial V}$. This is addressed in the next subsection.

\subsection{The moving particle lemma for the boundary-driven exclusion process} \label{sec:boundaryMPL}

In this subsection we prove the boundary-driven version of of Proposition \ref{prop:MPL}.

\begin{proposition}
\label{prop:boundaryMPL}
Suppose condition (\nameref{ass:boundaryrate}) holds. Then for all probability densities $f$ w.r.t.\@ $\nu_\lambda$, all $\alpha\in (0,1)$, and all $x,y\in V(G)$, we have
\begin{align}
\label{boundaryMPL}
& \quad \nu_\alpha\left[\sqrt{f \frac{d\nu_\lambda}{d\nu_\alpha}} \left(-\nabla_{xy}\sqrt{f\frac{d\nu_\lambda}{d\nu_\alpha}}\right)\right] \\ 
\nonumber & \leq  2 R^{(G,{\bf c})}_{\rm eff}(x,y) \left(\nu_\lambda\left[\sqrt f(-\mathcal{L}^{\rm EX}_{(G,{\bf c})}\sqrt f)\right] + \frac{1}{2}(\delta^{-1}-1)\left( \frac{1}{2\delta^2} \sum_{a\in \partial V} |i_\rho(a)| + \frac{1}{\delta^3} \mathcal{E}^{\rm el}_{(G,{\bf c})} (\rho) \right) \right),
\end{align}
where $\delta :=\frac{1}{1+\gamma} \in (0, \frac{1}{2}]$, and $i_\rho(a)$ is the electric flow out of $a\in \partial V$, \emph{cf.\@} \eqref{eq:flow}.
\end{proposition}


Observe that the last two terms on the RHS of \eqref{boundaryMPL}, as well as the effective resistance, are quantities associated with the random walk process on $(G,{\bf c})$. 

The key difference between Proposition \ref{prop:boundaryMPL} and Proposition \ref{prop:MPL} is as follows. While $\nu_\alpha$ is reversible for $\mathcal{L}^{\rm EX}_{(G,{\bf c})}$, $\nu_\lambda$ is \emph{not} reversible for $\mathcal{L}^{\rm EX}_{(G,{\bf c})}$ (though it is reversible for $\mathcal{L}^b_{\partial V}$). So while $\nu_\alpha\left[f\left(-\mathcal{L}^{\rm EX}_{(G,{\bf c})} f\right)\right]$ is always nonnegative, there is no reason for $\nu_\lambda\left[f\left(-\mathcal{L}^{\rm EX}_{(G,{\bf c})} f \right)\right]$ to be nonnegative. That said, we can estimate the difference between $\nu_\lambda\left[f\left(-\mathcal{L}^{\rm EX}_{(G,{\bf c})} f \right)\right]$ and the nonnegative quantity $\sum_{xy\in E(G)} c_{xy} \nu_\lambda\left[(\nabla_{xy} f)^2\right]$. The following result is a generalization of \cite{BDGJL03}*{Lemma 3.2} to weighted graphs.

\begin{lemma}
\label{lem:asymppos}
For any product Bernoulli measure $\nu$ on $\{0,1\}^{V(G)}$, and every $f\in L^2(\nu)$,
\begin{align}
\label{ineq:asymppos}
\nu\left[f\left(-\mathcal{L}^{\rm EX}_{(G,{\bf c})} f\right)\right] \geq \frac{1}{2} \sum_{xy\in E(G)} c_{xy} \nu\left[\left(\nabla_{xy} f\right)^2\right]-\frac{1}{2}\sup_{\eta\in \{0,1\}^{V(G)}} \left|\sum_{xy\in E(G)} \left(\frac{d\nu(\eta^{xy})}{d\nu(\eta)}-1\right) \right| \cdot \nu\left[f^2\right] .
\end{align}
\end{lemma}

\begin{proof}
It is straightforward to verify the identity
\begin{align}
\label{idbreak}
\nu\left[f\left(-\mathcal{L}^{\rm EX}_{(G,{\bf c})} f\right)\right] = \frac{1}{2} \sum_{xy\in E(G)} c_{xy} \nu\left[(\nabla_{xy} f)^2\right] + \frac{1}{2} \sum_{xy\in E(G)}c_{xy} \int\,\left([f(\eta^{xy})]^2 - [f(\eta)]^2\right)\, d\nu(\eta).
\end{align}
Since $\nu$ is a product Bernoulli measure, we can rewrite the intgral in the second term on the RHS as
\[
\int\, [f(\eta)]^2 \left(\frac{d\nu(\eta^{xy})}{d\nu(\eta)}-1\right)\,d\nu(\eta).
\]
Upon interchanging the summation and the integration and applying H\"{o}lder's inequality, we see that
\begin{align*}
\left|\sum_{xy\in E(G)} c_{xy}\int\, [f(\eta)]^2 \left(\frac{d\nu(\eta^{xy})}{d\nu(\eta)}-1\right)\,d\nu(\eta)\right| \leq \sup_{\eta\in \{0,1\}^{V(G)}} \left|\sum_{xy\in E(G)} c_{xy} \left(\frac{d\nu(\eta^{xy})}{d\nu(\eta)}-1\right) \right| \cdot \nu\left[f^2\right].
\end{align*}
The lemma follows from this and an application of the triangle inequality to \eqref{idbreak}.
\end{proof}

We will apply Lemma \ref{lem:asymppos} to the product Bernoulli measure $\nu_\lambda$. To prove the local ergodic theorem, we need to control the supremum term on the RHS of \eqref{ineq:asymppos}, using properties of discrete harmonic functions.

\begin{lemma}
\label{lem:RDPT}
Let $\delta \in (0, \frac{1}{2}]$, and $h: V(G)\to[\delta,1-\delta]$ be harmonic on $V(G)\setminus \partial V$. Denote by $\nu_h$ the product Bernoulli measure on $\{0,1\}^{V(G)}$ with marginals $\nu_h\{\eta: \eta(x)=1\} = h(x)$, $x\in V(G)$. Then
\begin{align}
\label{SigmaN}
\sup_{\eta\in \{0,1\}^{V(G)}} \left|\sum_{xy\in E(G)} c_{xy} \left(\frac{d\nu_h(\eta^{xy})}{d\nu_h(\eta)}-1 \right) \right| \leq \frac{1}{\delta^2} \sum_{a\in \partial V} |i_h(a)| + \frac{2}{\delta^3} \mathcal{E}^{\rm el}_{(G,{\bf c})}(h).
\end{align}
\end{lemma}

\begin{proof}
A direct but tedious calculation shows that for each $xy\in E(G)$,
\begin{align*}
\sum_{xy\in E(G)} c_{xy}\left(\frac{d\nu_h(\eta^{xy})}{d\nu_h(\eta)}-1\right) =: J_1 + J_2,
\end{align*}
where
\begin{align*}
J_1 &= \sum_{xy\in E(G)} \left(\eta(x) c_{xy} \frac{h(y)-h(x)}{h(x)[1-h(y)]} + \eta(y) c_{xy} \frac{h(x)-h(y)}{h(y)[1-h(x)]}\right),\\
J_2 &= \sum_{xy\in E(G)} \eta(x) \eta(y) c_{xy} \left[h(x)-h(y)\right]\left(\frac{1}{h(x)[1-h(y)]} - \frac{1}{h(y)[1-h(x)]}\right).
\end{align*}
By replacing the sum over edges with the double sums over vertices, we can rewrite $J_1$ as
\begin{align}
\label{eq:J1}
J_1 = \sum_{x\in V(G)} \eta(x) \sum_{y\in V(G)} c_{xy} \frac{h(y)-h(x)}{h(x)[1-h(y)]}.
\end{align}
To simplify the expression further, we use the identity
\begin{align}
\label{eq:hh}
\frac{h(y)-h(x)}{h(x)[1-h(y)]} &= \frac{h(y)-h(x)}{h(x)[1-h(x)]} + \frac{h(y)-h(x)}{h(x)}\left(\frac{1}{1-h(y)}- \frac{1}{1-h(x)}\right)\\
\nonumber &= \frac{h(y)-h(x)}{h(x)[1-h(x)]} +  \frac{[h(y)-h(x)]^2}{h(x)[1-h(x)][1-h(y)]}.
\end{align}
Plug (\ref{eq:hh}) into (\ref{eq:J1}) to get
\begin{align}
\label{J1breakdown}
J_1 &= \sum_{x\in V(G)} \eta(x) \sum_{y\in V(G)} c_{xy} \left[ \frac{h(y)-h(x)}{h(x)[1-h(x)]} + \frac{[h(y)-h(x)]^2}{h(x)[1-h(x)][1-h(y)]}\right] \\
\nonumber &=\sum_{x\in V(G)\setminus\partial V} \frac{\eta(x)}{h(x)[1-h(x)]} \sum_{y\in V(G)} c_{xy} [h(y)-h(x)] \\
& \nonumber \qquad + \sum_{x\in \partial V} \frac{\eta(x)}{h(x)[1-h(x)]} \sum_{y\in V(G)} c_{xy}[h(y)-h(x)] \\
& \nonumber \qquad + \sum_{x\in V(G)} \eta(x) \sum_{y\in V(G)} c_{xy} \frac{[h(y)-h(x)]^2}{h(x)[1-h(x)][1-h(y)]}\\
\nonumber &=: J_{11} + J_{12}+ J_{13}.
\end{align}

Since $h$ is harmonic on $V(G)\setminus \partial V$, $\sum_{y\in V(G)} c_{xy}[h(y)-h(x)]=0$ for each $x\in V(G)\setminus\partial V$. It follows that $J_{11}=0$. For $J_{12}$ we recall the definition of electric flow $i_h(x) := \sum_{y\in V(G)} c_{xy} [h(y)-h(x)]$. Thus
\[
J_{12} =\sum_{a\in \partial V} \frac{\eta(a)}{h(a)[1-h(a)]} i_h(a).
\]
Using the triangle inequality and that $h\in [\delta, 1-\delta]$ we obtain the estimate
\begin{align}
\label{J12}
|J_{12}| \leq \frac{1}{\delta^2} \sum_{a\in \partial V}|i_h(a)|.
\end{align}

For $J_{13}$, we once again use $h\in [\delta,1-\delta]$ to obtain
\begin{align}
\label{J13}
0\leq J_{13} &\leq \sum_{x\in V(G)} \sum_{y\in V(G)} c_{xy} \frac{[h(y)-h(x)]^2}{h(x)[1-h(x)][1-h(y)]} \\ 
\nonumber &\leq \frac{1}{\delta^3} \sum_{x\in V(G)}\sum_{y\in V(G)} c_{xy} [h(y)-h(x)]^2 = \frac{2}{\delta^3} \mathcal{E}^{\rm el}_{(G,{\bf c})}(h).
\end{align}

For $J_2$, we note the identity
\begin{align*}
\left[h(x)-h(y)\right]\left(\frac{1}{h(x)[1-h(y)]} - \frac{1}{h(y)[1-h(x)]}\right) = - \frac{[h(x)-h(y)]^2}{h(x)h(y)[1-h(x)][1-h(y)]}.
\end{align*}
This implies that $J_2\leq 0$.

Combining $J_{11}=0$, $J_2<0$, \eqref{J12}, \eqref{J13}, and applying the triangle inequality, yields
\begin{align*}
\text{LHS of \eqref{SigmaN}} &= |J_{11} + J_{12}+ J_{13}+ J_2| \leq |J_{12}| + |J_{13}+J_2| \leq |J_{12}|+J_{13}\\
& \leq \frac{1}{\delta^2} \sum_{a\in \partial V}|i_h(a)| + \frac{2}{\delta^3} \mathcal{E}^{\rm el}_{(G,{\bf c})}(h).
\end{align*}
This proves Lemma \ref{lem:RDPT}.
\end{proof}

We need one additional change-of-measure formula.

\begin{lemma}
\label{lem:timechange}
Let $h: V(G)\to [\delta, 1-\delta]$ and $\nu_h$ be as in the statement of Lemma \ref{lem:RDPT}. For every $\alpha\in (0,1)$ and every probability density $f$ w.r.t.\@ $\nu_h$,
\begin{align}
\label{ineq:timechange}
\sum_{xy\in E(G)} c_{xy} \nu_h\left[(\nabla_{xy}\sqrt{f})^2\right] &\geq \frac{1}{2} \sum_{xy\in E(G)} c_{xy} \nu_\alpha\left[\left(\nabla_{xy} \sqrt{f\frac{d\nu_h}{d\nu_\alpha}}\right)^2\right] \\
\nonumber & \quad - (\delta^{-1}-2)  \left( \frac{1}{\delta^2}\sum_{a\in \partial V} |i_h(a)| + \frac{2}{\delta^3} \mathcal{E}^{\rm el}_{(G,{\bf c})}(h)\right).
\end{align}
\end{lemma}
\begin{proof}
Let $g = f\frac{d\nu_h}{d\nu_\alpha} \in L^1(\nu_\alpha)$. Then
\begin{align}
\label{eq:string}
& \qquad \nu_\alpha\left[\left(\nabla_{xy} \sqrt{g}\right)^2\right] \\ 
\nonumber &= \int\, \left(\sqrt{g(\eta^{xy})}-\sqrt{g(\eta)}\right)^2 \frac{d\nu_\alpha}{d\nu_h}(\eta)\, d\nu_h(\eta) \\
\nonumber &= \int\, \left(\left(\sqrt{g\frac{d\nu_\alpha}{d\nu_h}(\eta^{xy})}-\sqrt{g\frac{d\nu_\alpha}{d\nu_h}(\eta)}\right) + \sqrt{g(\eta^{xy})}\left(\sqrt{\frac{d\nu_\alpha}{d\nu_h}(\eta)}-\sqrt{\frac{d\nu_\alpha}{d\nu_h}(\eta^{xy})} \right) \right)^2\, d\nu_h(\eta)\\
\nonumber &\leq 2 \int\, \left(\sqrt{g\frac{d\nu_\alpha}{d\nu_h}(\eta^{xy})} - \sqrt{g\frac{d\nu_\alpha}{d\nu_h}(\eta)}\right)^2\, d\nu_h(\eta) + 2\int\, g(\eta^{xy}) \left( \sqrt{\frac{d\nu_\alpha}{d\nu_h}(\eta)} - \sqrt{\frac{d\nu_\alpha}{d\nu_h}(\eta^{xy})}\right)^2\,d\nu_h(\eta) \\
\nonumber &= 2 \nu_h\left[\left(\nabla_{xy} \sqrt{f}\right)^2\right] + 2 \int\, g(\eta) \left( 1- \sqrt{\frac{d\nu_h(\eta^{xy})}{d\nu_h(\eta)} }\right)^2 \, d\nu_\alpha(\eta)
\end{align}
where in the last line, we applied to the second term a change of variables $\eta \to \eta^{xy}$, and used the fact that $\nu_\alpha$ charges the same measure to $\eta$ and to $\eta^{xy}$.

Next we estimate
\begin{align}
\label{eq:csq}
\sum_{xy \in E(G)} c_{xy} \left(1-\sqrt{\frac{d\nu_h(\eta^{xy})}{d\nu_h(\eta)}} \right)^2 = \sum_{xy\in E(G)} c_{xy} \left(1-\frac{d\nu_h(\eta^{xy})}{d\nu_h(\eta)}\right) Y_h(\eta, xy)
\end{align}
where
\[
Y_h(\eta,xy) := \frac{1-\sqrt{\frac{d\nu_h(\eta^{xy})}{d\nu_h(\eta)}}}{1+\sqrt{\frac{d\nu_h(\eta^{xy})}{d\nu_h(\eta)}}}.
\]
Since $h\in [\delta, 1-\delta]$ and $\delta\in (0,\frac{1}{2}]$,
\[
\frac{\delta^2}{(1-\delta)^2} \leq\frac{d\nu_h(\eta^{xy})}{d\nu_h(\eta)} \leq \frac{(1-\delta)^2}{\delta^2},
\]
which yields a uniform upper bound
\begin{align*}
 |Y_h(\eta,xy)| \leq \left|1-\sqrt{\frac{d\nu_h(\eta^{xy})}{d\nu_h(\eta)}}\right| \leq \delta^{-1}-2.
\end{align*}
This along with Lemma \ref{lem:RDPT} shows that \eqref{eq:csq} is bounded above by
\[
 (\delta^{-1}-2) \left( \frac{1}{\delta^2}\sum_{a\in \partial V} |i_h(a)| + \frac{2}{\delta^3} \mathcal{E}^{\rm el}_{(G,{\bf c})}(h)\right).
\]
Now combine this upper bound with the inequality \eqref{eq:string} to obtain
\begin{align*}
\sum_{xy\in E(G)} c_{xy} \nu_\alpha\left[(\nabla_{xy}\sqrt g)^2\right] &\leq 2 \sum_{xy\in E(G)} c_{xy} \nu_h\left[(\nabla_{xy}\sqrt f)^2\right] \\ & \quad +  2(\delta^{-1}-2) \left( \frac{1}{\delta^2}\sum_{a\in \partial V} |i_h(a)| + \frac{2}{\delta^3} \mathcal{E}^{\rm el}_{(G,{\bf c})}(h)\right).
\end{align*}
The lemma follows.
\end{proof}

\begin{proof}[Proof of Proposition \ref{prop:boundaryMPL}]
Let $f$ be a probability density relative to $\nu_\lambda$, whose one-site marginal $\nu_\lambda\{\eta: \eta(x)=1\}=\rho(x)$ is by Corollary \ref{cor:boundaryrate}.
Using Lemma \ref{lem:asymppos} (second line below), Lemmas  \ref{lem:RDPT} and \ref{lem:timechange} (third line), and Proposition \ref{prop:MPL} (fifth line), we obtain the estimate
\begin{align*}
& \quad \nu_\lambda\left[\sqrt f(-\mathcal{L}^{\rm EX}_{(G,{\bf c})} \sqrt f)\right] \\ 
\nonumber &\geq \frac{1}{2} \sum_{zw\in E(G)} c_{zw} \nu_\lambda\left[(\nabla_{zw} \sqrt{f})^2\right] - \frac{1}{2} \sup_{\eta\in \{0,1\}^{V(G)}}\left|\sum_{zw\in E(G)} \left(\frac{d\nu_\lambda(\eta^{zw})}{d\nu_\lambda (\eta)}-1\right)\right|\\
& \nonumber \geq \frac{1}{4} \sum_{zw\in E(G)} c_{zw} \nu_\alpha\left[\left(\nabla_{zw}\sqrt{f\frac{d\nu_\lambda}{d\nu_\alpha}}\right)^2\right] - \frac{1}{2} (\delta^{-1}-2)\left(\frac{1}{\delta^2} \sum_{a\in \partial V} |i_\rho(a)| + \frac{2}{\delta^3} \mathcal{E}^{\rm el}_{(G,{\bf c})} (\rho) \right) \\
& \nonumber \quad - \frac{1}{2}\left(\frac{1}{\delta^2} \sum_{a\in \partial V} |i_\rho(a)| + \frac{2}{\delta^3} \mathcal{E}^{\rm el}_{(G,{\bf c})} (\rho) \right) \\
& \nonumber \geq \frac{1}{4} \left(R_{\rm eff}^{(G,{\bf c})}(x,y)\right)^{-1} \nu_\alpha\left[\left(\nabla_{xy} \sqrt{f\frac{d\nu_\lambda}{d\nu_\alpha}}\right)^2 \right] - \frac{1}{2}(\delta^{-1}-1)\left(\frac{1}{\delta^2} \sum_{a\in \partial V} |i_\rho(a)| + \frac{2}{\delta^3} \mathcal{E}^{\rm el}_{(G,{\bf c})} (\rho) \right).
\end{align*}
Now apply the identity $\frac{1}{2}\nu_\alpha\left[(\nabla_{xy} g)^2\right] = \nu_\alpha\left[g(-\nabla_{xy} g) \right]$, $g\in L^2(\nu_\alpha)$, to the first term on the RHS. A simple algebraic manipulation yields \eqref{boundaryMPL}.
\end{proof}

\subsection{Finishing the proof} \label{sec:localergodicboundary}


\begin{proof}[Proof of Theorem \ref{thm:localergodicboundary}]
The proof is based on a variation of the methods in \S\ref{sec:1block} and \S\ref{sec:2blocks}. We will make all estimates w.r.t.\@ $\mathbb{P}^N_{\nu^N_\lambda}$, where the product Bernoulli measure $\nu^N_\lambda$ on $\{0,1\}^{V(\Gamma_N)}$ has one-site marginals $\nu^N_\lambda\{\eta: \eta(x)=1\} = \rho_N(x)$, $x\in V(\Gamma_N)$. At the end we transfer the estimates to $\mathbb{P}^N_{\eta_0^N}$.

\emph{One-block estimate, \emph{cf.\@} \S\ref{sec:1block}.} First, we have the analog of Lemma \ref{lem:DFEst1}: for every $\Lambda_1 \subset \Lambda_2 \subset V(\Gamma)$ and every $f\in L^2(\nu_\lambda^{\Lambda_2})$,
\begin{align}
\label{ineq:comp}
\frac{1}{2} \sum_{\substack{xy\in E(\Gamma)\\x,y\in \Lambda_1}}\nu_\lambda^{\Lambda_1}\left[(\nabla_{xy} f)^2 \right] \leq \frac{1}{2} \sum_{\substack{xy\in E(\Gamma)\\x,y\in \Lambda_2}}\nu_\lambda^{\Lambda_2}\left[(\nabla_{xy} f)^2\right].
\end{align}

Let $\left(\mathcal{L}^{\rm bEX}_{(\Gamma_N, {\bf c})}\right)^*$ (resp.\@ $\left(\mathcal{L}^{\rm EX}_{(\Gamma_N, {\bf c})}\right)^*$) be the adjoint of $\mathcal{L}^{\rm bEX}_{(\Gamma_N, {\bf c})}$ (resp.\@ $\mathcal{L}^{\rm EX}_{(\Gamma_N, {\bf c})}$) relative to the nonreversible measure $\nu_\lambda^N$. Observe that $\nu_\lambda^N$ is reversible for the symmetrized generator 
\begin{align*}
\frac{1}{2}\left(\mathcal{L}^{\rm bEX}_{(\Gamma_N,{\bf c})} + \left(\mathcal{L}^{\rm bEX}_{(\Gamma_N,{\bf c})}\right)^*\right) = \frac{1}{2} \left(\mathcal{L}^{\rm EX}_{(\Gamma_N, {\bf c})} + \left(\mathcal{L}^{\rm EX}_{(\Gamma_N, {\bf c})}\right)^*\right) + \mathcal{L}^b_{\partial V_N}.
\end{align*}

We modify Lemma \ref{lem:1beigv} as follows: for each $\kappa>0$, let $\lambda^{(1),\pm}_{N,j,\lambda}(\kappa)$ be the largest eigenvalue of 
\begin{align*}
&~\quad\frac{1}{2} \left( \left[\mathcal{T}_N \mathcal{L}^{\rm bEX}_{(\Gamma_N,{\bf c})} \pm \kappa \mathcal{V}_N \Uone{j}(\pt,\cdot)\right] + \left[\mathcal{T}_N \mathcal{L}^{\rm bEX}_{(\Gamma_N,{\bf c})} \pm \kappa \mathcal{V}_N \Uone{j}(\pt,\cdot)\right]^* \right) \\
\nonumber &= \frac{\mathcal{T}_N}{2} \left( \mathcal{L}^{\rm EX}_{(\Gamma_N,{\bf c})} + \left( \mathcal{L}^{\rm EX}_{(\Gamma_N,{\bf c})}\right)^*\right) + \mathcal{T}_N \mathcal{L}^b_{\partial V_N} \pm \kappa \mathcal{V}_N \Uone{j}(\pt,\cdot)
\end{align*}
 with respect to $\nu^N_\lambda$. 
 We claim that there exists a constant $c_\delta$ such that for all $\kappa>0$,
 \begin{align}
 \label{1beigv2}
 \limsup_{j\to\infty} \limsup_{N\to\infty}\frac{1}{\kappa \mathcal{V}_N} \lambda^{(1),\pm}_{N,j,\lambda}(\kappa) \leq \frac{c_\delta}{\kappa}.
 \end{align}
To see this, we again use the variational characterization
\begin{align}
\label{boundarysup}
\frac{1}{\kappa \mathcal{V}_N} \lambda^{(1),\pm}_{N,j,\lambda}(\kappa)&=\sup_f \left\{ \nu^N_\lambda\left[\pm \Uone{j}(\pt,\cdot)f\right]\right.\\
\nonumber & \qquad - \left.\frac{\mathcal{T}_N}{\kappa\mathcal{V}_N} \left(\nu^N_\lambda\left[\sqrt{f} \left(-\mathcal{L}^{\rm EX}_{(\Gamma_N,{\bf c})} \sqrt{f}\right) \right] + \nu_\lambda^N \left[ \sqrt{f} \left(-\mathcal{L}^b_{\partial V_N} \sqrt{f}\right)\right]\right) \right\},
\end{align}
where the supremum is taken over all densities relative to $\nu^N_\lambda$. 
Using \eqref{ineq:comp}, $\nu_\lambda^N \left[ \sqrt{f} \left(-\mathcal{L}^b_{\partial V_N} \sqrt{f}\right)\right]\geq 0$, and Lemmas \ref{lem:asymppos} and \ref{lem:RDPT}, we can bound (\ref{boundarysup}) from above by
\begin{align}
\label{b1b}
&\sup_f \left\{\nu^{\Lambda_j(x)}_\lambda\left[\pm \Uone{j}(x,\cdot)f\right] -  \frac{1}{2}\frac{\mathcal{T}_N}{\kappa \mathcal{V}_N} \sum_{\substack{zw\in E(\Gamma)\\ z,w\in \Lambda_j(x)}} c_{zw} \nu^{\Lambda_j(\pt)}_\lambda\left[\left(\nabla_{zw} f\right)^2\right]\right.\\
\nonumber & \qquad \qquad \qquad +\left.\frac{\mathcal{T}_N}{\kappa\mathcal{V}_N} \left(\frac{1}{2\delta^2} \sum_{a\in \partial V_N} |i_{\rho_N}(a)| + \frac{1}{\delta^3} \mathcal{E}^{\rm el}_{(\Gamma_N,{\bf c})}(\rho_N) \right)\right\},
\end{align}
where $f$ is any probability density w.r.t.\@ $\nu^{\Lambda_j(\pt)}_\lambda$.

Now take the limit of \eqref{b1b} as $N\to\infty$, and note that we can interchange the limit and the supremum by a compactness argument. By Assumption \ref{ass:1} (resp.\@ Assumption \ref{ass:boundaryscaling}), the second term (resp.\@ the third term) of \eqref{b1b} converges to $-\infty$ (resp.\@ $c_\delta/\kappa$ for a positive constant $c_\delta$). Thus
\begin{align*}
\limsup_{N\to\infty} \frac{1}{\kappa \mathcal{V}_N} \lambda^{(1),\pm}_{N,j,\lambda}(\kappa) \leq \sup_f \limsup_{N\to\infty}\left\{\nu_\lambda^{\Lambda_j(x)}\left[\pm \Uone{j}(\pt,\cdot)f\right] \right\} + \frac{c_\delta}{\kappa}.
\end{align*}
Next take the limit $j\to\infty$. The first term on the RHS vanishes according to the equivalence of ensembles argument, \emph{cf.\@} the end of the proof of Lemma \ref{lem:1beigv} followed by Lemma \ref{lem:ensembles}. Thus \eqref{1beigv2} is proved. 

From \eqref{1beigv2}, we use the exponential Chebyshev's inequality and the Feynman-Kac formula as in the proof of Theorem \ref{thm:1block} to complete the one-block estimate w.r.t.\@ $\mathbb{P}^N_{\nu^N_\lambda}$.

\emph{Two-blocks estimate, \emph{cf.\@} \S\ref{sec:2blocks}.} 
Let $\lambda^{(2),\pm}_{N,j,x,y,\lambda}(\kappa)$ be the largest eigenvalue of
\begin{align*}
\frac{1}{2}\left(\left[\mathcal{T}_N \mathcal{L}^{\rm bEX}_{(\Gamma_N,{\bf c})} \pm \kappa \mathcal{V}_N \Utildetwo{j}{x}{y}(\cdot)\right]+ \left[\mathcal{T}_N \mathcal{L}^{\rm bEX}_{(\Gamma_N,{\bf c})} \pm \kappa \mathcal{V}_N \Utildetwo{j}{x}{y}(\cdot)\right]^*\right).
\end{align*}
w.r.t.\@ $\nu_\lambda^N$. We claim that there exists a positive constant $c_{\delta}$ such that
\begin{align}
\label{b2b}
\limsup_{j\to\infty}\limsup_{\epsilon \downarrow 0}\limsup_{N\to\infty} \frac{1}{\kappa \mathcal{V}_N} \lambda^{(2),\pm}_{N,j,x,y,\lambda}(\kappa) \leq \frac{c_{\delta}}{\kappa}.
\end{align}
The proof of \eqref{b2b} follows nearly identically the proof presented in \S\ref{sec:2blocks}, but with one key change due to the nonreversibility of $\mathcal{L}^{\rm bEX}_{(\Gamma_N, {\bf c})}$ w.r.t.\@ $\nu_\lambda^N$. 
Using \eqref{ineq:comp} and Lemmas \ref{lem:asymppos}, \ref{lem:RDPT}, and \ref{lem:timechange}, we find that for every probability density $f$ w.r.t.\@ $\nu_\lambda$,
\begin{align*}
& \quad \nu_\alpha\left[\sqrt{f\frac{d\nu_\lambda}{d\nu_\alpha}}\left( -\left(\mathcal{L}^{\rm EX}_{(\Lambda_j(x),{\bf c})} + \mathcal{L}^{\rm EX}_{(\Lambda_j(y), {\bf c})}\right)\sqrt{f \frac{d\nu_\lambda}{d\nu_\alpha}}\right)\right]  \\
& \leq \frac{1}{2} \sum_{zw\in E(G)} c_{zw} \nu_\alpha\left[\left(\nabla_{zw}\sqrt{f\frac{d\nu_\lambda}{d\nu_\alpha}}\right)^2 \right]\\
& \leq 2\nu_\lambda\left[\sqrt{f}\left(-\mathcal{L}^{\rm EX}_{(G,{\bf c})} \sqrt{f}\right) \right] + (\delta^{-1}-1)\left(\frac{1}{\delta^2} \sum_{a\in \partial V} |i_\rho(a)| + \frac{2}{\delta^3} \mathcal{E}^{\rm el}_{(G,{\bf c})} (\rho) \right).
\end{align*}
Harkening back to the arguments leading up to \eqref{ineqmoving}, but using Proposition \ref{prop:boundaryMPL} instead of Proposition \ref{prop:MPL}, we obtain
\begin{align*}
\nu_\alpha&\left[\sqrt{f\frac{d\nu_\lambda}{d\nu_\alpha}}\left(-\mathcal{L}^{(2)}_{j,x,y}\sqrt{ f\frac{d\nu_\lambda}{d\nu_\alpha}}\right) \right] 
\leq 2\left(1+\sum_{i=0}^{B-1} R_{\rm eff}^{(\Gamma_N, {\bf c})}(z_i, z_{i+1})\right) \\
 & \qquad \times\left( \nu_\lambda\left[\sqrt{f}(-\mathcal{L}^{\rm EX}_{(\Gamma_N,{\bf c})} \sqrt{f})\right] +\frac{1}{2}(\delta^{-1}-1)\left(\frac{1}{\delta^2} \sum_{a\in \partial V} |i_\rho(a)| + \frac{2}{\delta^3} \mathcal{E}^{\rm el}_{(G,{\bf c})} (\rho) \right)\right).
\end{align*}
This yields the estimate
\begin{align}
\label{2bboundary}
\frac{1}{\kappa \mathcal{V}_N}\lambda^{(2),\pm}_{N,j,x,y,\lambda}(\kappa) & \leq \sup_f\left\{\nu^{\Lambda^{(2)}(j,x,y)}_\lambda\left[\pm \tilde{U}^{(2)}_{j,x,y}(\cdot)f\right]-\frac{K_1}{2\kappa} \right\} + \frac{K_2}{2\kappa},
\end{align}
where
\begin{align*}
K_1 &:= \frac{\mathcal{T}_N}{\mathcal{V}_N}\left(1+ \sum_{i=0}^{B-1} R_{\rm eff}^{(\Gamma_N, {\bf c})}(z_i, z_{i+1})\right)^{-1}\nu_\alpha^{\Lambda^{(2)}(j,x,y)}\left[\sqrt{f\frac{d\nu_\lambda}{d\nu_\alpha}}\left(-\mathcal{L}^{(2)}_{j,x,y} \sqrt{f\frac{d\nu_\lambda}{d\nu_\alpha}}\right)\right] ,\\
K_2 &:= (\delta^{-1}-1)\frac{\mathcal{T}_N}{\mathcal{V}_N}\left(\frac{1}{\delta^2} \sum_{a\in \partial V_N} |i_{\rho_N}(a)| + \frac{2}{\delta^3} \mathcal{E}^{\rm el}_{(G,{\bf c})} (\rho_N)\right),
\end{align*}
and the supremum runs over all probability densities w.r.t.\@ $\nu_\lambda^{\Lambda^{(2)}(j,x,y)}$.
Upon taking the limit $N\to\infty$ then $\epsilon\downarrow 0$ then $j\to\infty$ on \eqref{2bboundary}, the supremum on the RHS vanishes by the same argument presented in \S\ref{sec:2blocks}, while $K_2$ tends to a finite positive number $2 c_{\delta}$ by  Assumption \ref{ass:boundaryscaling}. This proves \eqref{b2b}.

From \eqref{b2b} we complete the two-blocks estimate w.r.t.\@ $\mathbb{P}^N_{\nu_\lambda^N}$ using the exponential Chebyshev's inequality and the Feynman-Kac formula as before. This leads to
\begin{align}
\label{bdsup}
\limsup_{\epsilon\downarrow 0} \limsup_{N\to\infty} \sup_\pt \frac{1}{\mathcal{V}_N} \log \mathbb{P}^N_{\nu^N_\lambda} \left\{ \int_0^T\, \left|U_{N,\epsilon}(\pt,\eta^N_t)\right|\,dt>\delta\right\} = -\infty.
\end{align}

\emph{Change of measure.}
Finally we transfer the estimates from $\nu_\lambda^N$ to an initial configuration $\eta_0^N$.
Note the Radon-Nikod\'ym derivative
\[
 \frac{d \eta^N_0}{d\nu_\lambda^N} =\mathbbm{1}_{\eta_0^N}\cdot\exp\left(\sum_{x\in V(\Gamma_N)} \left[\eta_0^N(x) \log\left(\frac{1}{\rho_\lambda^N(x)} \right) + (1-\eta_0^N(x))\log\left(\frac{1}{1-\rho_\lambda^N(x)}\right) \right]\right).
\]
Under Condition (\nameref{ass:boundaryrate}), $\rho_\lambda^N(x) \in \left[\frac{1}{1+\gamma}, \frac{\gamma}{1+\gamma}\right]$ for all $x$ and $N$, so there exists a positive constant $C_\gamma := \log(1+\gamma)$ such that for all $N$,
\begin{align*}
\left\|\frac{d \eta_0^N}{d\nu_\lambda^N}\right\|_\infty \leq e^{C_\gamma |V(\Gamma_N)|}.
\end{align*}
Thus for any measurable event $\mathcal{O}$,
\begin{align}
\label{ineqO}
\frac{1}{\mathcal{V}_N} \log \mathbb{P}^N_{\eta_0^N}[\mathcal{O}] \leq \frac{1}{\mathcal{V}_N} \log\mathbb{P}^N_{\nu^N_\lambda}[\mathcal{O}] + C_\gamma \frac{|V(\Gamma_N)|}{\mathcal{V}_N}.
\end{align}
By Condition (\nameref{ass:boundedaway}), the second term on the RHS is asymptotically bounded above. From\eqref{bdsup} and \eqref{ineqO} we obtain \eqref{supexpboundary}.

To prove \eqref{supexpboundary2}, we set
\[
U(a,\eta) := \eta(a)-\frac{\lambda_+(a)}{\lambda_+(a)+\lambda_-(a)}, \quad a\in \partial V_N, ~\eta\in \{0,1\}^{V(\Gamma_N)},
\]
and consider the largest eigenvalue of
\[
\frac{1}{2}\left(\left[\mathcal{T}_N\mathcal{L}^{\rm bEX}_{(\Gamma_N,{\bf c})} \pm \kappa \mathcal{V}_N U(a,\cdot)\right]+\left[\mathcal{T}_N\mathcal{L}^{\rm bEX}_{(\Gamma_N,{\bf c})} \pm \kappa \mathcal{V}_N U(a,\cdot)\right]^*\right)
\]
w.r.t.\@ the measure $\nu_\lambda^N$. We claim that this eigenvalue is finite. After running the aforementioned variational argument, it boils down to checking that
\[
\lim_{N\to\infty} \nu_\lambda^{\partial V_N}[U(a,\eta)] = \lim_{N\to\infty} \left[\rho_N(a) - \frac{\lambda_+(a)}{\lambda_+(a)+\lambda_-(a)} \right]=0
\]
But this is \eqref{eq:ratecoincide}, which follows from the condition \eqref{flowfinite} in Assumption \ref{ass:boundaryscaling}. The remaining arguments (applying exponential Chebyshev's inequality, the Feynman-Kac formula, and a change of measure) are routine, yielding \eqref{supexpboundary2}.
\end{proof}

\section{Examples}
\label{sec:examples}

In this section we comment in more detail on the assumptions, stated in \S\ref{sec:main}, which underlie our local ergodic theorems (both the conservative version and the boundary-driven version), and provide examples where the assumptions apply. 

Throughout this section, the notation $f(\cdot) \simeq g(\cdot)$ means that there exist positive constants $c\leq C$ such that $cg(\cdot) \leq f(\cdot) \leq Cg(\cdot)$ uniformly in the argument.

\subsection{Comparison between mean exit time and mean commute time}

As alluded to in \S\ref{sec:superexp}, for many applications we take $\mathcal{V}_N$ to stand for $|B(o, r_N)|$ or $\mathcal{V}(o, r_N)$, and $\mathcal{T}_N$ to be the (extremal) mean exit time from $B(o,r_N)$. In this situation, Assumption \ref{ass:1} amounts to the condition that the random walk is \emph{strongly recurrent}, a property carried by many a weighted graph whose \emph{spectral dimension} is less than (but not equal to) $2$; this will be explained precisely in \S\ref{sec:strongrec} below. Meanwhile, Assumption \ref{ass:2} boils down to an asymptotic comparison between the mean exit time and the mean commute time of the random walk.

\begin{lemma}
If $\mathcal{V}_N = \mathcal{V}(o, r_N)$, then Assumption \ref{ass:2} is equivalent to the following condition:
\begin{align}
\label{ratiocommute}
\liminf_{\epsilon \downarrow 0} \liminf_{N\to\infty} \frac{\mathcal{T}_N}{\tc(y,z)} = \infty
\end{align}
for each $x\in V(\Gamma)$ and each $y,z\in B(x,r_{\epsilon N})$, where 
\[
\tc(y,z):={\bf E}^y[T_z] + {\bf E}^z[T_y]
\]
 is the \textbf{mean commute time} between $y$ and $z$ for the symmetric random walk on $(\Gamma_N,{\bf c})$.
 
\end{lemma}
\begin{proof}
We recall the \emph{mean commute time identity} \cite{CommuteTime} (see also \cite{MCBook}*{Proposition 10.6}): given a finite weighted graph $(G,{\bf c})$, we have that for every $y,z \in V(G)$,
\begin{align}
\label{eq:meancommute}
{\bf E}^y[T_z] + {\bf E}^z[T_y] = \mathcal{V}(G) R^{(G,{\bf c})}_{\rm eff}(y,z).
\end{align}
Since we assumed $\Gamma$ is locally finite, \eqref{eq:meancommute} applies to our setting with $(G,{\bf c})= (\Gamma_N, {\bf c})$. Thus we can use \eqref{eq:meancommute} to reexpress \eqref{eq:ass2} in Assumption \ref{ass:2} as \eqref{ratiocommute}.
\end{proof}

It is generally difficult to infer the asymptotics of commute times from those of hitting times; the analysis often proceeds on a case-by-case basis. For more discussions on this issue we refer the reader to \emph{e.g.\@} \cite{MCBook}*{Chapter 10} and \cite{AldousFill}*{Chapters 4--5}.

\subsection{Strongly recurrent weighted graphs} \label{sec:strongrec}

For the next three subsections we always assume

\begin{condition}[$p_0$]
\label{cond:ellipticity}
There exists a constant $p_0>0$ such that 
\begin{align}
p(x,y) :=\frac{c_{xy}}{c_x} \geq p_0 \quad \text{for all}~x\in V(\Gamma)~\text{and}~xy\in E(\Gamma).
\end{align}
\end{condition}

Random walk on a connected weighted graph $(\Gamma,{\bf c})$ is said to be \emph{recurrent} if, for every $x\in V(\Gamma)$, ${\bf P}^x[T_x^+ <\infty]=1$. This is a notion familiar from the theory of Markov chains. \emph{Strong recurrence} is a slightly more restrictive notion than recurrence. From the existing literature we are aware of (at least) two formulations of strong recurrence.

\begin{definition}[\cites{Telcs01,Telcs01_2}]
\label{def:SR}
$(\Gamma, {\bf c})$ is \textbf{strongly recurrent (SR)} if there exist $c>0$ and $M>1$ such that
\begin{align}
R_{\rm eff}\left(x, B(x,Mr)^c\right) \geq (1+c) R_{\rm eff}\left(x, B(x,r)^c\right) \quad \text{for all}~x\in V(\Gamma), ~r\geq 1.
\end{align}
\end{definition}

\begin{definition}[\cites{Delmotte,BarlowValues}]
\label{def:VSR}
$(\Gamma, {\bf c})$ is \textbf{very strongly recurrent (VSR)} if there exists $p_1>0$ such that
\begin{align}
{\bf P}^x\left[T_y< \mathcal{T}(x,2r)\right] \geq p_1 \quad \text{for all}~x\in V(\Gamma),~r\geq 1, ~d(x,y) < r.
\end{align}
\end{definition}

It is proved in \cite{BCK05}*{Lemma 3.6} that (VSR) implies (SR). See \cite{BCK05}*{\S5, Example 5} for an example of a weighted graph which satisfies (SR) but not (VSR). What distinguishes (VSR) from (SR) is the addition of the elliptic Harnack inequality.

\begin{definition}
$(\Gamma, {\bf c})$ satisfies the \textbf{elliptic Harnack inequality (H)} if there exists $C>0$ such that for all $x\in V(\Gamma)$, all $r>1$, and all nonnegative harmonic functions $h$ on $B(x,2r)$,
\begin{align}
\max_{B(x,r)} h \leq C \min_{B(x,r)} h.
\end{align}
\end{definition}


\begin{proposition}
\label{prop:VSRSR}
Under ($p_0$), (VSR) is equivalent to (SR) $+$ (H).
\end{proposition}
\begin{proof}
We already mentioned that (VSR) implies (SR). That (VSR) implies (H) is proved in \cite{BarlowValues}*{Lemma 1.6}. 
For the reverse direction see \cite{TelcsBook}*{\S9.3.1}.
\end{proof}

Our main result of this section is that local ergodicity in the exclusion process holds under ($p_0$) and (VSR).

\begin{proposition}
\label{SRprop}
($p_0$) and (VSR) implies Assumptions \ref{ass:1} and \ref{ass:2}.
\end{proposition}
\begin{proof}
Let $A$ be any measurable subset of $V(\Gamma)$. Using the Green's function $G^A$ killed upon exiting $A$, the reversibility of the random walk process, and a last exit decomposition, we obtain
\[
\mathbb{E}_x[T_{A^c}] = \sum_{y\in A} G^A(x,y) = \sum_{y\in A} \frac{\mathcal{V}(y)}{\mathcal{V}(x)}G^A(y,x)  = \sum_{y\in A} \frac{\mathcal{V}(y)}{\mathcal{V}(x)} \mathbb{P}_y[T_x<T_A] G^A(x,x) = \frac{G^A(x,x) }{\mathcal{V}(x)} \mathcal{V}(A).
\]
Hence
\[
\frac{\mathbb{E}_x[T_{A^c}]}{\mathcal{V}(A)} = \frac{G^A(x,x)}{\mathcal{V}(x)}.
\]
Assuming that random walks on $(\Gamma, {\bf c})$ are recurrent, $G^A(x,x) \to \infty$ as $A\uparrow \Gamma$. This verifies Assumption \ref{ass:1}.

Next we check Assumption \ref{ass:2}. By Proposition \ref{prop:VSRSR}, ($p_0$) $+$ (VSR) is equivalent to ($p_0$) $+$ (SR) $+$ (H). By \cite{Telcs01_2}*{Corollary 4.6}, ($p_0$), (SR) and (H) implies that
\begin{align}
\label{TelcsScaling}
\mathcal{T}(x,r) \simeq m(B(x,r)) R_{\rm eff}^{(\Gamma, {\bf c})}\left(x, B(x,r)^c\right) \quad \text{for all}~x\in V(\Gamma),~r\geq 1.
\end{align}
For any $y,z \in B(x, r_{\epsilon N})$,
\begin{align}
\frac{\mathcal{T}_N}{\mathcal{V}_N} \left(R_{\rm eff}^{(\Gamma_N,{\bf c})}(y,z)\right)^{-1} \simeq \frac{R_{\rm eff}^{(\Gamma, {\bf c})}\left(o,B(o,r_N)^c\right)}{R_{\rm eff}^{(\Gamma_N,{\bf c})}(y,z)} \geq C \frac{R_{\rm eff}^{(\Gamma, {\bf c})}\left(o,B(o,r_N)^c\right)}{\rho(B(x,r_{\epsilon N}))},
\end{align}
where $\rho(A)$ stands for the diameter of $A\subset V(\Gamma)$ w.r.t.\@ the effective resistance metric $R_{\rm eff}^{(\Gamma, {\bf c})}$. Lastly, observe that for every $N$, $\rho(B(x,r_{\epsilon N}))$ monotonically decreases to $0$ as $\epsilon \downarrow 0$. So
\begin{align}
\liminf_{\epsilon\downarrow 0} \liminf_{N\to\infty} \frac{R_{\rm eff}^{(\Gamma, {\bf c})}\left(o,B(o,r_N)^c\right)}{\rho(B(x,r_{\epsilon N}))} = \infty.
\end{align}
Assumption \ref{ass:2} is thus verified.
\end{proof}

%

\begin{remark}
\eqref{TelcsScaling} is reminiscent of the \emph{Einstein relation}, which connects the expected exit time of a random walk with the volume and the resistance of the graph. There is some variation in how the Einstein relation is defined in the literature. In \cite{KigamiMemoir} the point-to-point effective resistance is used:
\[
\mathcal{T}(x,r) \simeq  \mathcal{V}(B(x,r))\sup_{y\in B(x,r)} R_{\rm eff}^{(\Gamma,{\bf c})}(x,y) \quad \text{for all}~x\in V(\Gamma), ~r\geq 1.
\]
 In \cite{TelcsBook} the volume of the ball is replaced by the volume of an annulus, \emph{cf.\@} \cite{TelcsBook}*{(7.1)}:
\[
\mathcal{T}(x, 2r) \simeq \mathcal{V}\left(B(x,2r) \setminus B(x,r)\right) R_{\rm eff}^{(\Gamma, {\bf c})}\left(B(x,r), B(x,2r)^c\right) \quad \text{for all}~x\in V(\Gamma),~r\geq 1.
\]
\end{remark}

The next result concerns the range of the volume growth exponent $\alpha$ and of the time growth exponent $\beta$ of weighted graphs to which Assumptions \ref{ass:1} and \ref{ass:2} apply. We recall the conditions
\begin{align}
\tag{$V_\alpha$} \text{There exists $C\geq 1$ such that} \quad  C^{-1} r^\alpha \leq &\mathcal{V}(B(x,r)) \leq C r^\alpha \quad \text{for all}~x\in V(\Gamma), ~r\geq 1\\
\tag{$E_\beta$} \text{There exists $c\geq 1$ such that} \quad  c^{-1} r^\beta \leq &\mathcal{T}(x,r)\leq c r^\beta \quad \text{for all}~ x\in V(\Gamma), ~r\geq 1.
\end{align}

\begin{proposition}
\label{prop:BarlowValues}
Let $\alpha \geq 1$ and 
\[
\beta \in \left\{\begin{array}{ll} [2, 1+\alpha], & \text{if}~\alpha \in [1,2),\\ (\alpha, 1+\alpha], &\text{if}~\alpha \in [2,\infty).\end{array} \right.
\]
Then there exists an infinite connected locally finite graph which satisfies $(V_\alpha)$, $(E_\beta)$, (H), and (VSR), and thus also satisfy Assumptions \ref{ass:1} and \ref{ass:2}.
\end{proposition}

See Figure \ref{fig:BarlowValues} for the range indicated by Proposition \ref{prop:BarlowValues}.

\begin{proof}
By \cite{BarlowValues}*{Theorem 2}, if $\alpha\geq 1$ and $2\leq \beta \leq 1+\alpha$, then there exists an infinite connected locally finite graph which satisfies $(V_\alpha)$, $(E_\beta)$, and (H). (In fact, the converse is also true: if an infinite connected weighted graph satisfies $(p_0)$, ($V_\alpha$), and ($E_\beta$), then $\alpha\geq 1$ and $2\leq \beta \leq 1+\alpha$; \emph{cf.\@} \cite{BarlowValues}*{Theorem 1}.) Furthermore, by \cite{BarlowValues}*{Proposition 3}, if $(V_\alpha)$, $(E_\beta)$ and (H) hold, and $\beta>\alpha$, then $\Gamma$ satisfies (VSR). Now use Proposition \ref{SRprop}.
\end{proof}

\begin{figure}
\centering
\includegraphics[width=0.5\textwidth]{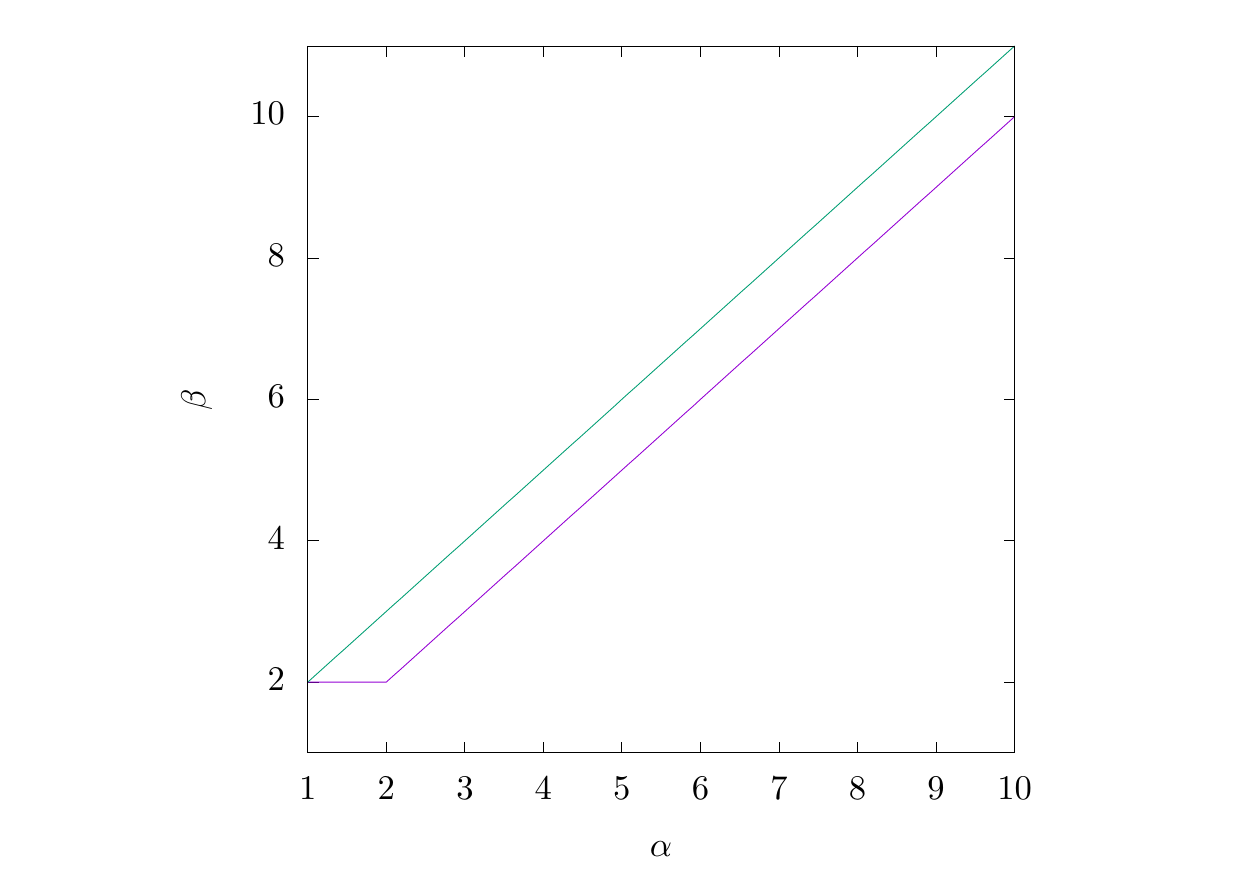}
\caption{The range of volume growth exponent $\alpha$ and of the time growth exponent $\beta$ for which Assumptions \ref{ass:1} and \ref{ass:2} hold is bounded between the two lines, \emph{cf.\@} Proposition \ref{prop:BarlowValues}.}
\label{fig:BarlowValues}
\end{figure}

\subsection{Euclidean lattices}

We would be remiss not to mention the case of $\mathbb{Z}^d$, regarded as a weighted graph with nearest neighbor edges with all conductances equal to $1$. Clearly ($p_0$) holds, and it is a classical result that (H) holds. As for the type problem, it is very strongly recurrent when $d=1$, recurrent but not strongly recurrent when $d=2$, and transient when $d\geq 3$. Therefore by Proposition \ref{SRprop}, Assumptions \ref{ass:1} and \ref{ass:2} hold simultaneously on $\mathbb{Z}$, but not on $\mathbb{Z}^d$, $d\geq 2$. That being said, local ergodicity has been proved on $\mathbb{Z}^d$ for any dimension $d$ using the invariance of the Dirichlet form $\mathcal{E}^{\rm EX}_{\mathbb{Z}^d}(f)$ under lattice translations and rotations; see \cites{GPV88, KOV89,KipnisLandim}.

\subsection{Examples of strongly recurrent graphs}

In this subsection we discuss specific examples of strongly recurrent graphs.
This also allows us to check Assumption \ref{ass:boundaryscaling} on a case-by-case basis. 

\subsubsection{Sierpinski gasket (and other post-critically finite fractals)}

\begin{figure}
\centering
\includegraphics[width=0.8\textwidth]{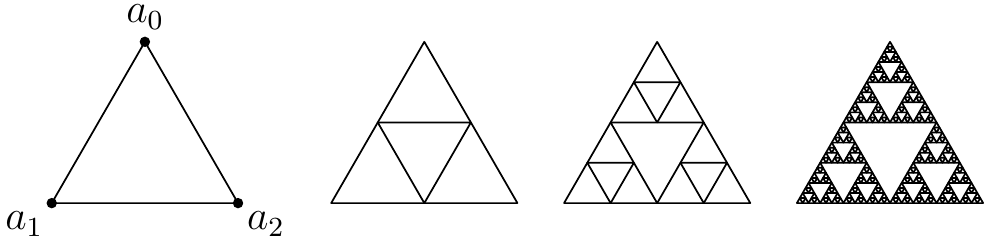}
\caption{The Sierpinski gasket ($SG$) graphs of level $0$, $1$, $2$, and $5$, respectively.}
\label{fig:SGImage}
\end{figure}

We quickly recall the construction of the infinite Sierpinski gasket graph; see Fgure \ref{fig:SGImage}. Let $a_0$, $a_1$, $a_2$ be the vertices of a nondegenerate triangle in $\mathbb{R}^2$, and $G_0$ be the complete graph on the vertex set $V_0 = \{a_0, a_1, a_2\}$, as shown on the left in Figure \ref{fig:SGImage}. We declare $V_0$ to be the (analytical but not topological) boundary of $SG$. Define the contracting similitude $\Psi_i : \mathbb{R}^2 \to\mathbb{R}^2$, $\Psi_i(x) = \frac{1}{2}(x-a_i) + a_i $ for each $i\in \{0,1,2\}$. For each $N\in \mathbb{N}$, the $N$th-level $SG$ graph $G_N$ is constructed inductively via the formula $G_N = \bigcup_{i=0}^2 \Psi_i(G_{N-1})$. Finally, set the conductance on every $e\in E_N$ to $1$. We denote the corresponding weighted graph $(G_N, {\bf 1})$.

For each $m$-letter word $w=w_1 w_2 \cdots w_m \in \{0,1,2\}^m$ , put $\Psi_w = \Psi_{w_1} \circ \Psi_{w_2} \circ \cdots \circ \Psi_{w_m}$. Two vertices $x, y\in V_N$ are said to be in the same level-$j$ cell, $j\in \{0,1,\cdots, N\}$, if $x, y\in \Psi_w(V_0)$ for some $j$-letter word $w \in \{0,1,2\}^j$.


We now proceed to check Assumptions \ref{ass:1} and \ref{ass:2} on $SG$. An explicit approach is as follows. Set $o$ at the origin (the lower-left corner vertex) and take $r_N = 2^N$. It is well known that $\mathcal{V}_N \simeq 3^N$ and $\mathcal{T}_N \simeq 5^N$, and hence Assumption \ref{ass:1} is satisfied. Also, the effective resistance metric and the graph metric on $\Gamma$ satisfy the scaling $R_{\rm eff}^{(\Gamma,{\bf c})}(y,z) \simeq d(y,z)^\beta$ where $\beta = \frac{\log(5/3)}{\log 2}$. So for all $y, z\in B(x, r_{\epsilon N})$,
\begin{align}
\frac{\mathcal{T}_N}{\mathcal{V}_N} \left(R_{\rm eff}^{(\Gamma_N,{\bf c})}(y,z)\right)^{-1} \simeq \frac{5^N}{3^N} (2^{\beta \epsilon N})^{-1} = \left(\frac{5}{3}\right)^{(1-\epsilon)N},
\end{align}
which tends to $\infty$ as $N\to\infty$ for every $\epsilon\in [0,1)$. Thus Assumption \ref{ass:2} is satisfied. Altogether this implies that Theorem \ref{thm:localrep} holds.

Alternatively, one can invoke the fact that $SG$ is a (very) strongly recurrent graph in the sense of Definition \ref{def:SR} or \ref{def:VSR}, and apply Proposition \ref{SRprop}.

For the boundary-driven case, we fix for all $N$ the rates $\lambda_\pm(a_i) = \rho_i \in (0,\infty)$ for $i=0,1,2$. It is straightforward to check conditions (\nameref{ass:boundedaway}) and (\nameref{ass:boundaryrate}), so the key remaining assumption to check is Assumption \ref{ass:boundaryscaling}. Let us note that for $a\in V_0$ and $u: V(\Gamma_N)\to\mathbb{R}$, $i_u(a)$ coincides with (the negative of) the normal derivative
\begin{align}
(\partial_\perp^N u)(a) = \sum_{\substack{x\in V(\Gamma_N) \\ ax\in E(\Gamma_N)}} [u(a)-u(x)]. 
\end{align}
By \cite{Kigami}*{Lemma 3.7.7} or \cite{StrichartzBook}*{Theorem 2.3.2}, if $f \in {\rm dom}(\Delta_\mu)$, where $\Delta_\mu$ is the Laplacian on $SG$ w.r.t.\@ any suitable measure $\mu$ (\emph{cf.\@} \cite{Kigami}*{Definition 3.7.1}), then the limit $(\frac{5}{3})^N (\partial^N_\perp f|_{\Gamma_N})(a)$ exists. Since $\rho_N$ is the restriction to $\Gamma_N$ of the harmonic function $\rho$, which belongs to ${\rm dom}(\Delta_\mu)$ with $\mu$ being the standard self-similar measure on $SG$, it follows that the limit $\lim_{N\to\infty} (\frac{5}{3})^N i_{\rho_N}(a)$ exists, and thus $\lim_{N\to\infty} (\frac{5}{3})^N \sum_{a \in \partial V_N} |i_{\rho_N}(a)| <\infty$ as $\partial V_N$ is a finite set. Also it is elementary to check that $\mathcal{E}^{\rm el}_{(\Gamma_N,{\bf c})}(\rho_N) \simeq \left(\frac{3}{5}\right)^N$. This concludes the verification of Assumption \ref{ass:boundaryscaling}. Thus Theorem \ref{thm:localergodicboundary} holds.

The aforementioned arguments extend to a wider class of fractals called \emph{post-critically finite (p.c.f.\@) fractals}, see \cite{Kigami}*{Definition 1.3.13} for the definition. As p.c.f.\@ fractals forms a subfamily of strongly recurrent fractals, Assumptions \ref{ass:1} and \ref{ass:2} will apply. In the case of the boundary-driven exclusion process, where the reservoirs are placed on the post-critical set in the construction of the p.c.f.\@ fractal (see \cite{Kigami}*{Definition 1.3.4}), Assumption \ref{ass:boundaryscaling} can be verified using \cite{Kigami}*{Lemma 3.7.7}, essentially a Gauss-Green formula.

\subsubsection{Sierpinski carpets}

\begin{figure}
\centering
\includegraphics[width=0.25\textwidth]{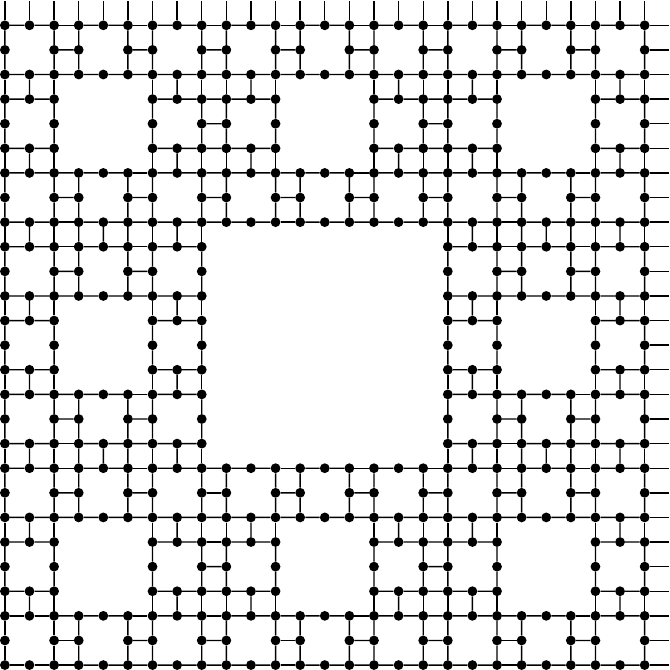}
\caption{A portion of the infinite Sierpinski carpet ($SC$) graph.}
\label{fig:SCGraph}
\end{figure}

An example of an infinitely ramified fractal is the Sierpinski carpet, which can be realized as an infinite graph as shown in Figure \ref{fig:SCGraph}. It is a strongly recurrent weighted graph. As a result, Assumptions \ref{ass:1} and \ref{ass:2} hold there. The validity of these assumptions extend to all \emph{generalized} Sierpinski carpet graphs (for the precise definition see \cite{BBKT}*{\S2.2}) whose spectral dimension is less than $2$. 

Unfortunately, at this moment we are unable to verify Assumption \ref{ass:boundaryscaling}, and in turn, establish a local ergodic theorem in the boundary-driven exclusion process, on the Sierpinski carpet graph. Unlike on p.c.f.\@ fractals, the corresponding boundary theory on the Sierpinski carpet remains a challenging problem; see \cites{HinoKumagai, BKS13} for known results. 

\subsubsection{Trees}

There are many examples of trees to which our Assumptions \ref{ass:1} and \ref{ass:2} are applicable. On the other hand, there are also many other trees which do not satisfy our Assumptions, such as $d$-regular trees with equal weights. Here we focus on two types of trees to which our Assumptions apply. A more systematic analysis is possible by extending the calculations of \cite{Pearce}.

\textbf{Vicsek trees.} These highly symmetric trees arising from the construction of the Vicsek set are described in \emph{e.g.\@} \cite{BarlowValues}*{\S4}; see Figure \ref{fig:Vicsek} for an example of a Vicsek tree. By \cite{BarlowValues}*{Lemma 4.2}, Vicsek trees satisfy ($V_\alpha$), ($E_{\alpha+1}$), (H), and (VSR). Thus they form the upper bound in the $\beta$ values possible for each given $\alpha\geq 1$, as indicated in Proposition \ref{prop:BarlowValues} (see the top line in Figure \ref{fig:BarlowValues}).

\begin{figure}
\centering
\includegraphics[width=0.25\textwidth]{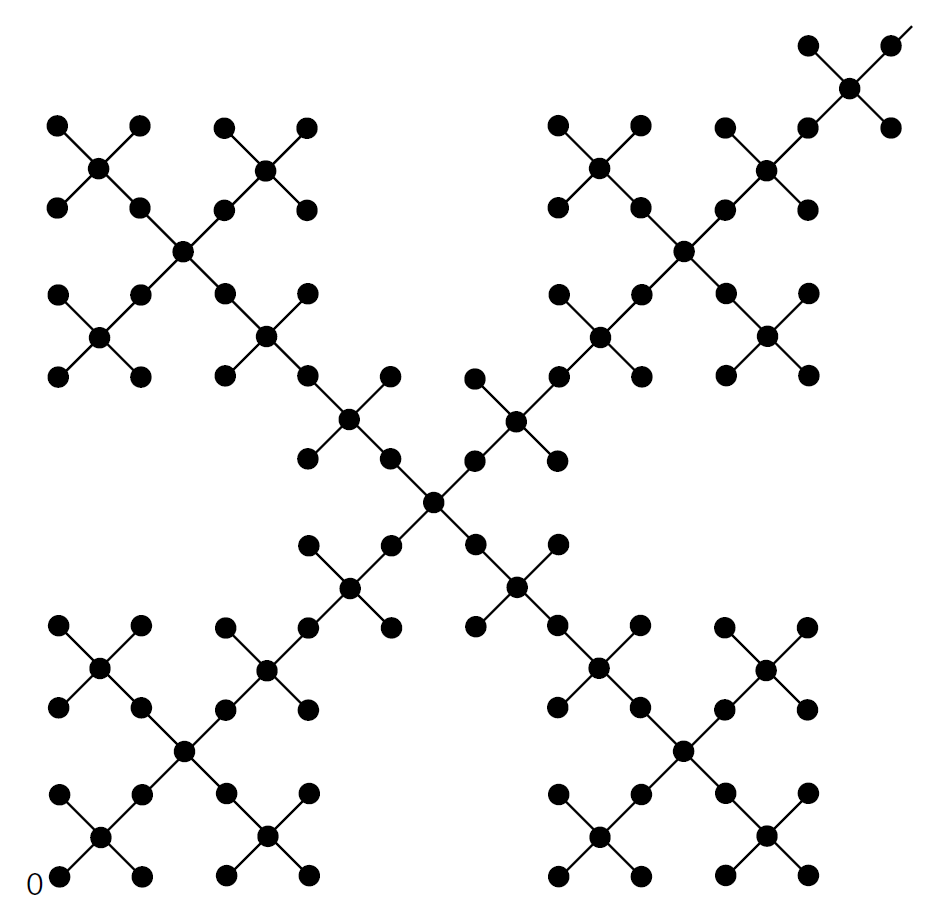}
\caption{The Vicsek tree.}
\label{fig:Vicsek}
\end{figure}

\textbf{The continuum random tree (CRT).} The CRT was introduced by Aldous in \cites{AldousCRT1, AldousCRT2, AldousCRT3}. Croydon and Hambly \cite{CH08} showed that the CRT can be realized as a random self-similar fractal represented by the metric measure space $(\mathcal{T}, d_\mathcal{T}, \mu)$. Specifically they proved the following result \cite{CH08}*{Theorem 2.1}: Almost surely there exists a local regular Dirichlet form $(\mathcal{E}_\mathcal{T}, \mathcal{F}_\mathcal{T})$ on $L^2(\mathcal{T}, \mu)$, which is associated with the metric $d_\mathcal{T}$ such that for every $x\neq y$,
\[
d_{\mathcal{T}}(x,y)^{-1} = \inf\{\mathcal{E}_\mathcal{T}(f,f): f\in \mathcal{F}_\mathcal{T},~f(x)=1,~f(y)=0\}.
\] 
This indicates that $d_\mathcal{T}$ is the resistance metric on the CRT, and on a tree, the resistance metric is equal to the geodesic metric. It is proved in \cite{DuquesneLeGall} that the  CRT has Hausdorff dimension (w.r.t.\@ the resistance metric) equal to $2$, although it is not uniformly volume doubling (due to logarithmic fluctuations), see \cite{Croydon08}*{Theorems 1.1$\sim$1.3} for the relevant estimates in both the quenched and annealed cases. From heat kernel estimates established in \cite{Croydon08} or the Weyl spectral asymptotitcs established in \cite{CH08}, one can infer that CRT has walk dimension $3$.

\begin{remark}
The results above may also be obtained informally as follows. Many finitely ramified fractals have the property that
\begin{align}
\label{dSdH}
\frac{d_S}{2} = \frac{d_H}{d_H+1},
\end{align}
where $d_H$ is the Hausdorff dimension of the fractal w.r.t.\@ the resistance metric, and $d_S$ is the \emph{spectral dimension}. This is shown in \cite{KigamiLapidus}*{Theorem A.2}. On a tree, assuming that ($V_\alpha$), ($E_\beta$), and the Einstein relation ($\frac{d_S}{2} = \frac{\alpha}{\beta}$) hold, \eqref{dSdH} implies that $\beta=\alpha+1$, which corresponds to the upper bound for the admissible exponents in Proposition \ref{prop:BarlowValues}.
\end{remark}

\subsubsection{Random graphs arising from percolation}

Over the past decade there has been significant progress in the study of random walks on percolation clusters on $\mathbb{Z}^d$ (the ``ant in the labyrinth'' problem). Taking these clusters as our random graphs, one can obtain heat kernel estimates and deduce the spectral dimensions (\emph{cf.\@} the Alexander-Orbach conjecture). It turns out that many such random graphs are strongly recurrent: for instance, the incipient infinite cluster (IIC) on a tree, as well as a high-dimensional (oriented) percolation cluster, satisfies (with high probability) a volume condition and a resistance condition amounting to a volume growth exponent $2$ and time growth exponent $3$. Details are discussed in the monograph \cite{KumagaiStFlour}, in particular (3.22) and Theorem 4.1 therein.

\section{Open questions} \label{sec:open}

We conclude the paper with some outstanding questions and future directions.

\textbf{Connections to scaling limits under the Gromov-Hausdorff-vague convergence.} In a recent preprint \cite{CHK16} Croydon, Hambly, and Kumagai established the convergence of diffusion processes on a sequence of metric measure spaces $(F_N, R_N, \mu_N, \rho_N)$ ($\rho_N$ is a distinguished point in $F_N$) which converges to $(F,R,\mu,\rho)$ in the \emph{Gromov-Hausdorff-vague topology} (its distinction from the Gromov-Hausdorff topology is explained in \cite{ALW16}) and satisfies a uniform volume doubling condition. Croydon \cite{Croydon} then obtained the said convergence of diffusion processes without assuming uniform volume doubling, rather replacing it with the simpler resistance condition, \emph{cf.\@} \cite{Croydon}*{Assumption 1.1}:
\begin{align}
\label{ass:Croydon}
\lim_{r\to\infty} \limsup_{N\to\infty} R_N(\rho_N, B_N(\rho_N, r)^c) =\infty,
\end{align}
where $B_N(\rho_N, r) := \{y\in F_N: R_N(\rho_N,y)<r\}$.
We believe that our Assumptions \ref{ass:1} and \ref{ass:2} are strongly related to Croydon's condition \eqref{ass:Croydon}, and may be unified via the language of Kigami's resistance forms \cites{KigamiRF,KigamiMemoir}. This may potentially lead to a unified treatment of convergence of stochastic processes (both for the single-particle diffusion process and the multi-particle exclusion process) on a resistance space (\emph{i.e.,} space endowed with a resistance form).

\textbf{``Slowing'' the boundary reservoirs.} Take the 1D boundary-driven exclusion process on $\{1,2,\cdots, N-1\}$ subject to rate-$1$ particle exchanges at the two reservoirs at $0$ and $N$. One can make the boundaries ``slow'' by allowing particles to be created or annihilated at a rate proportional to $N^{-\theta}$ at the boundary vertices, where $\theta$ is any positive real number. In \cite{BMNS} it is shown that depending on $\theta$, the scaling limit of the empirical density is the unique solutions to the heat equation on the unit interval with one of these boundary condition: Dirichlet ($\theta\in (0,1)$), Robin ($\theta =1$), or Neumann ($\theta \in (1,\infty)$). 

As a consequence of Assumption \ref{ass:boundaryscaling}, our present work only addresses the Dirichlet boundary condition. We believe that it is feasible to modify the argument presented in \S\ref{sec:boundary} so as to obtain scaling limits with Robin or Neumann boundary conditions on a resistance space.

\textbf{Local ergodicity on non-strongly-recurrent weighted graphs.} A major outstanding problem is to establish the local ergodic theorem without Assumption \ref{ass:1}, or more concretely, on weighted graphs which support transient random walks. The same question is also open for spaces with spectral dimension equal to $2$.

\textbf{Exclusion processes on (self-similar) groups.}
Since random walks on groups \cite{Woess} are predominantly of transient type, our methodology does not extend readily to symmetric exclusion processes on groups. Nevertheless we would like to mention a potential connection.

Tanaka \cite{Tanaka} recently proved a local ergodic theorem on Cayley graphs associated to a class of infinite amenable groups which are quasi-transitive w.r.t\@ some group actions. As an application, his methods may be adapted to establish local ergodicity on a close analog of $SG$, the so-called Hanoi-tower graph (or the stretched Sierpinski gasket, $SSG$ \cite{AFK16}), which is the Schreier graph of a self-similar group \cite{NekBook}. We believe that his methodology, in conjunction with the moving particle lemma of \cite{ChenMPL}*{Theorem 1} (replacing \cite{Tanaka}*{Lemma 3.1}), should allow one to obtain the two-blocks estimate on $SSG$. (The condition $\sqrt{t_m} \leq 2\,{\rm diam}X_m$ stated in \cite{Tanaka}*{p.\@ 5} is equivalent to the condition that the random walks satisfy Gaussian or \emph{super}-Gaussian space-time estimates. Hence it does \emph{not} apply to fractal graphs like $SSG$.) The connection between Tanaka's work and analysis on self-similar groups \cite{NekTep} is a subject of future study.

\textbf{Generalization to the zero-range process.}
At the moment we do not know how to generalize our machinery from the exclusion process to the \emph{zero-range process} (which is defined in \emph{e.g.\@} \cite{KipnisLandim}*{\S2.3}). The biggest obstacle is the lack of a functional inequality (like the octopus inequality \cite{CLR09}*{Theorem 2.3}) which would imply, for instance, an optimal spectral gap or a moving particle lemma (like \cite{ChenMPL}*{Theorem 1}). So it remains elusive to establish a two-blocks estimate for the zero-range process on an infinite weighted graph.

That said, if one is only interested in obtaining a law of large numbers for the empirical density, then it is possible to bypass the two-blocks estimate. On $SG$ Jara \cite{Jara} combined the local one-block estimate of \cite{JLSLocal}, Yau's $H_{-1}$ method, and precise Green's function estimates to prove the convergence of the empirical density to the unique weak solution of a nonlinear heat equation.

\subsection*{Acknowledgements}

I am indebted to Alexander Teplyaev for his valuable insights and assistance which helped shape the article in its present form, especially on the analysis presented in \S\ref{sec:boundary}.

This paper was completed while I was participating in the Institut Henri Poincar\'e (IHP) trimester program ``Stochastic Dynamics Out of Equilibrium.'' I would like to thank the IHP and the trimester organizers for providing a productive work environment, and Patricia Gon\c calves for her elaboration of the work \cite{BMNS} and several illuminating conversations.

\end{document}